\documentclass[usenames,dvipsnames]{article}

\usepackage{a4wide}

\usepackage{amsmath,amsfonts,amsthm,amssymb}
\usepackage[justification=justified, singlelinecheck=false]{caption}
\usepackage{graphicx}
\usepackage{enumitem}

\usepackage{tikz}

\usepackage{algorithmicx}
\usepackage{algorithm}
\usepackage{algpseudocode}

\usepackage[scaled=.90]{helvet}
\usepackage{courier}
\usepackage{ae}
\usepackage[T1]{fontenc}
\usepackage{subfigure}
\usepackage{graphicx,color}
\usepackage{xcolor}

\usepackage[T1]{fontenc}
\usepackage{here} 
\usepackage[colorlinks=false, pdfborder={0 0 0}]{hyperref}

\usepackage{xpatch}
\makeatletter
\xpatchcmd{\@thm}{\thm@headpunct{.}}{\thm@headpunct{}}{}{}
\makeatother

\newcommand{\norm}[1]{\left\lVert#1\right\rVert}

\newtheorem{theorem}{Theorem}
\newtheorem{definition}{Definition}
\newtheorem{lemma}{Lemma}
\newtheorem{assumption}{Assumption}
\newtheorem{notation}{Notation}
\newtheorem{remark}{Remark}
\newtheorem{proposition}{Proposition}
\newtheorem{corollary}{Corollary}
\theoremstyle{definition}

{\begin{proof}[Beweis]}
	{\end{proof}}

\begin{document}
	
\title{A splitting method for SDEs with locally Lipschitz drift:\\Illustration on the FitzHugh-Nagumo model}
%\date{October 31, 475}
%\author{Evelyn Buckwar, Adeline Samson, Massimiliano Tamborrino, Irene Tubikanec\\}
\author{Evelyn Buckwar\footnotemark[1]\thanks{Institute of Stochastics, Johannes Kepler University Linz (evelyn.buckwar@jku.at,irene.tubikanec@jku.at)} \footnotemark[2]\thanks{Centre for Mathematical Sciences, Lund University},  Adeline Samson\footnotemark[3]\thanks{Laboratoire Jean Kuntzmann, University Grenoble Alpes (Adeline.Leclercq-Samson@univ-grenoble-alpes.fr)}, Massimiliano Tamborrino\footnotemark[4]\thanks{Department of Statistics, University of Warwick (massimiliano.tamborrino@warwick.ac.uk)}, Irene Tubikanec\footnotemark[1]\\}
\date{}
\maketitle	
	
%\tableofcontents
\thispagestyle{empty}

\section*{Abstract}

In this article, we construct and analyse an explicit numerical splitting method for a class of semi-linear stochastic differential equations (SDEs) with additive noise, where the drift is allowed to grow polynomially and satisfies a global one-sided Lipschitz condition. The method is proved to be mean-square convergent of order $1$ and to preserve important structural properties of the SDE. First, it is hypoelliptic in every iteration step. Second, it is geometrically ergodic and has an asymptotically bounded second moment. Third, it preserves oscillatory dynamics, such as amplitudes, frequencies and phases of oscillations, even for large time steps. Our results are illustrated on the stochastic FitzHugh-Nagumo model and compared with known mean-square convergent tamed/truncated variants of the Euler-Maruyama method. The capability of the proposed splitting method to preserve the aforementioned properties may make it applicable within different statistical inference procedures. In contrast, known Euler-Maruyama type methods commonly fail in preserving such properties, yielding ill-conditioned likelihood-based estimation tools or computationally infeasible simulation-based inference algorithms.

\subsubsection*{Keywords} Stochastic differential equations, Locally Lipschitz drift, Hypoellipticity, Ergodicity, FitzHugh-Nagumo model, Splitting methods, Mean-square convergence

\subsubsection*{AMS subject classifications} 60H10, 60H35, 65C20, 65C30

\subsubsection*{Acknowledgements} 

E.B. was supported by the LCM -- K2 Center within the framework of the Austrian COMET-K2 program. A.S. was supported by MIAI@Grenoble Alpes, (ANR-19-P3IA-0003). E.B., M.T. and I.T. were supported by the Austrian Science Fund (FWF): W1214-N15, project DK14. All authors were supported by the Austrian Exchange Service (OeAD), bilateral project FR 03/2017.

\section{Introduction}
\vspace{-0.2cm}

The aim of this article is to construct and analyse a splitting method for semi-linear stochastic differential equations (SDEs) of additive noise type
\begin{equation}\label{eq:semi_linear_SDE}
	dX(t)=F(X(t)) dt + \Sigma dW(t)\, := \, \bigl[ AX(t) + N(X(t)) \bigr]dt +\Sigma dW(t), \quad X(0)=X_0,  %\quad t \in [0,T],
\end{equation}
where the diffusion matrix $\Sigma$ may be degenerate and the drift $F$ satisfies a global one-sided Lipschitz condition and is allowed to grow polynomially. Coefficients with these properties appear in many applications \cite{Hutzenthaler2012}, ranging from physics \cite{Mattingly2002,Milstein2007} over population growth problems \cite{Hutzenthaler2007,Khasminskii2011} to neuroscience \cite{FitzHugh1961,Hodgkin1952,Nagumo1962} and others. 
As an illustrative equation from this class of SDEs, we discuss the stochastic FitzHugh-Nagumo (FHN) model \cite{Berglund2012,Bonaccorsi2008,Leon2018}, a well-known neuronal model describing the generation of spikes of single neurons at the intracellular level. This model is given by the $2$-dimensional SDE
\begin{equation}\label{FHN}
	d \underbrace{\begin{pmatrix}
			V(t) \\
			U(t) 
	\end{pmatrix}}_{:=X(t)}
	=
	\underbrace{\begin{pmatrix}
			\frac{1}{\epsilon}\Bigl(V(t)-V^3(t)-U(t)\Bigr) \\
			\gamma V(t)-U(t)+\beta
	\end{pmatrix}}_\text{$:=F(X(t))$} dt \ + \
	\underbrace{\begin{pmatrix}
			\sigma_1 & 0 \\
			0 & \sigma_2
	\end{pmatrix}}_{:=\Sigma} dW(t), %\quad X(0)=X_0, %\quad t \in [0,T], 
\end{equation}
where $\sigma_1$ may be zero. The $V$-component of the system describes the evolution of the membrane voltage of the neuron and the $U$-component is a recovery variable. Our aim is to construct a numerical method for \eqref{eq:semi_linear_SDE}, and \eqref{FHN} in particular, which is easy to implement and also applicable across different disciplines in the broad field of statistical inference. This implies that the method needs to meet several requirements: 

\begin{itemize}
	\item Statistical applications require strong approximations of SDEs. Thus, we focus on the concept of \textit{mean-square convergence} \cite{Kloeden1992,Milstein2004,Tretyakov2012}. Since it was shown in \cite{Hutzenthaler2011} that the standard Euler-Maruyama method does not converge in the mean-square sense under the above assumptions on the drift $F$, the development of mean-square convergent  variants of this method has received much attention. In particular, tamed \cite{Hutzenthaler2012_2,Sabanis2013,Tretyakov2012,Zhang2017} and truncated \cite{Mao2017,Hutzenthaler2012,Mao2015,Mao2016} Euler-Maruyama methods have been proposed. They all aim to control the unbounded growth arising from the non-globally Lipschitz drift by enforcing a rescaling modification to the drift and/or diffusion coefficients. 
	\item Simulation-based statistical methods require to generate paths of SDEs as computationally efficient as possible, see, e.g., \cite{Buckwar2019}. Using \textit{explicit} numerical methods  is a first step to achieve sufficiently low computational cost. While the aforementioned mean-square convergent variants of the Euler-Maruyama method are explicit, they commonly fail in preserving important structural properties of the SDE. 
	The major key to computational efficiency, however, is to construct explicit methods which are capable to preserve the underlying properties  for time steps as large as possible. This leads to the next point.
	\item An important issue in the field of (stochastic) numerical analysis is the \textit{preservation of structural properties} of the considered SDE by the numerical method used to approximate~it. Geometric Numerical Integration is a well-established framework in this context \cite{Hairer2006}. Here, we discuss the preservation of hypoellipticity, geometric ergodicity and oscillatory dynamics such as amplitudes, frequencies and phases of oscillations:
	\begin{itemize}
		\item The diffusion matrix $\Sigma$ of SDE \eqref{eq:semi_linear_SDE} may be of full rank or degenerate, where in the latter case the SDE may be \textit{hypoelliptic}, depending on the drift $F$. The case of degenerate noise naturally occurs in many applications \cite{Ableidinger2017,Ditlevsen2019,Leimkuhler2015,Leon2018,Mattingly2002,Milstein2007}, with the hypoelliptic property ensuring that the solution of the SDE admits a smooth transition density \cite{Nualart1995}. This means that the noise is propagated through the whole system via the drift of the SDE, even though it does not directly act on all components. In many inference approaches using discrete approximations of SDEs, it is necessary that a discrete analogue of the hypoelliptic property holds at each iteration step. In particular, considering a discretised time interval with equidistant time steps $\Delta=t_i-t_{i-1}$, the distribution of the numerical solution $\widetilde{X}(t_i)$ of \eqref{eq:semi_linear_SDE} at time $t_i$ given the previous value $\widetilde{X}(t_{i-1})$ must admit a smooth density, a property that we term $1$-step hypoellipticity. It is known that Euler-Maruyama type methods do not satisfy this if the SDE is not elliptic but only hypoelliptic. Thus, they yield ill conditioned likelihood-based inference methods \cite{Ditlevsen2019,Melnykova2018,Pokern2009}. Higher-order Taylor approximation methods \cite{Kloeden1992} may be $1$-step hypoelliptic \cite{Ditlevsen2019}. However, since such methods may neither be mean-square convergent in the case of superlinearly growing coefficients~\cite{Hutzenthaler2011}, nor preserve other structural properties, they may lead again to ill-posed statistical problems. 
		\item The analysis of the asymptotic behaviour of the process is of further crucial interest. In particular, if SDE \eqref{eq:semi_linear_SDE} possesses an underlying Lyapunov structure, it may be \textit{geometrically ergodic} \cite{Mattingly2002}. This property ensures that the distribution of the process converges exponentially fast to a unique limit for any starting value $X_0$, and has two important statistical implications. First, the choice of the initial  value $X_0$ is negligible since its impact on the distribution of the process decreases exponentially fast. This is relevant, especially when the process is only partially observed, since $X_0$ is usually not known. Second, there is a correspondence of ``time averages along trajectories'' and ``space averages across trajectories'' of geometrically ergodic systems, see, e.g., \cite{Ableidinger2017,daPrato1996}. This means that quantities related to the distribution of the process can be estimated from a single path simulated over a sufficiently large time horizon instead of relying on repeated simulations of trajectories. For the importance of this feature in statistical inference algorithms, we refer, e.g., to \cite{Buckwar2019}.
		Euler-Maruyama type methods may not provide these features as they tend to lose the Lyapunov structure of the SDE \cite{Ableidinger2017,Mattingly2002}. In particular, here we illustrate that they react sensitively to the initial condition $X_0$, and that they may yield poor approximations of the underlying invariant density~of~the~process.
		\item The last structural properties we are focusing on are features linked to oscillatory dynamics such as \textit{amplitudes, frequencies} and \textit{phases} of oscillations. Already in the deterministic scenario it has been observed that Euler type methods may not preserve amplitudes and frequencies of oscillations, see, e.g., \cite{Stern2020,Hairer2006}. Similar findings have been made for Euler-Maruyama type methods in the stochastic case. For example, it has been proved that the Euler-Maruyama method does not preserve the  growth rate of the second moment of linear stochastic harmonic oscillators, overshooting the amplitudes of the underlying oscillations, even for arbitrarily small time steps $\Delta$ \cite{Strommen2004}. Similar non-preserving results of oscillation amplitudes have been observed for non-linear, ergodic and higher-dimensional stochastic oscillators \cite{Ableidinger2017,Chevallier2020,Cohen2012}. Taming/truncating perturbations do not improve this behaviour. Even worse, taming perturbations may also lead to a non-preservation of frequencies of oscillations \cite{Kelly2017,Kelly2018}. This lack of amplitude and frequency preservation is also confirmed by our numerical experiments on the FHN model \eqref{FHN}. Moreover, we find that Euler-Maruyama type methods may also not preserve phases of oscillations. This poor behaviour is linked to the non-preservation of geometric ergodicity.  
	\end{itemize}
\end{itemize}

Here, we propose to apply the \textit{splitting} technique, an approach that addresses all previously listed issues. The general idea of this method is to split the SDE of interest into exactly solvable subequations, to derive their solutions, and to compose them in a suitable way. We refer to \cite{Blanes2009,Mclachlan2002} for a thorough discussion of splitting methods for ordinary differential equations (ODEs) and to \cite{Ableidinger2016,Ableidinger2017,Alamo2016,BouRabee2017,Brehier2018,Chevallier2020,Leimkuhler2015,Milstein2003,Misawa2001,Petersen1998,Shardlow2003} for articles considering extensions to SDEs. Note that it is often possible to split the differential equation under consideration into different sets of subequations, the choice of the most useful set depending on the problem to be solved. 
For the class of SDEs with additive noise, where the drift $F$ consists of a linear and a non-linear term, i.e., $F(X(t))=AX(t)+N(X(t))$, as in Equation \eqref{eq:semi_linear_SDE}, the idea is to exclude the nonlinear term $N(X(t))$ into a deterministic differential subequation and to treat the remaining linear term $AX(t)$ via a stochastic differential subequation. In this article, we exploit properties of the exact solutions of both subequations, as illustrated on the FHN model~\eqref{FHN}. When $N$ is globally Lipschitz continuous and uniformly bounded, this idea has been applied to the Jansen and Rit neural mass model in~\cite{Ableidinger2017}, and for locally Lipschitz $N$ it has been applied to the Allen-Cahn equation in \cite{Brehier2018}. 

We illustrate that splitting methods may be able to deal with mean-square convergence issues arising from superlinearly growing coefficients. In particular, in Lemma \ref{Lemma3}, we prove the boundedness of the moments of the proposed splitting method. This result is the key to establish its \textit{mean-square convergence}. In Theorem \ref{thm:convergence_FHN}, we use Tretyakov's and Zhang's ``Fundamental mean-square convergence theorem for SDEs with locally Lipschitz coefficients'' \cite{Tretyakov2012} to prove that the splitting method converges with mean-square order $1$. This is in agreement with the convergence rate of comparable known splitting methods in the globally Lipschitz scenario \cite{Ableidinger2017,Alamo2016,Milstein2003}, and with standard methods such as the Euler-Maruyama method in the case of additive noise \cite{Kloeden1992}. Moreover, we address the fact that splitting methods may also be able to tackle the problems arising from degenerate noise structures. In particular, in Theorem~\ref{thm:hypo_N_LT}, we show that the proposed splitting method is \textit{$1$-step hypoelliptic} and  yields non-degenerate multivariate normal transition distributions for any time step $\Delta$, provided that the stochastic subequation of the splitting framework is chosen to be hypoelliptic. This may be beneficial for likelihood-based inference. Furthermore, in Theorem \ref{thm:Lyapunov_discrete}, we prove that the constructed method satisfies a discrete Lyapunov condition, and is thus \textit{geometrically ergodic}, for any time step $\Delta$. This result requires an assumption on the solution of the deterministic subequation defined via $N(X(t))$ and that $\norm{e^{A\Delta}}<1$, where the matrix norm is induced by the Euclidean norm. Moreover, in Corollary \ref{cor:asym_2nd_moment}, we show that the second moment of the splitting method is asymptotically bounded by a constant which is independent of the time step size $\Delta$ and the number of time steps $i$. This result holds if, in addition, the logarithmic norm \cite{Soderlind2006,Strom1975} of the matrix $A$ is strictly negative. In the one-dimensional case, some of the involved expressions simplify such that, in Corollary \ref{cor:bounds_1dim}, we obtain a precise \textit{closed-form (asymptotic) bound of the second moment} of the proposed splitting method. This bound is illustrated on a cubic one-dimensional model problem with drift given by $F(X(t))=-X^3(t)$ \cite{Hutzenthaler2011,Mattingly2002}. In addition, we illustrate the proposed splitting method on the stochastic FHN model \eqref{FHN} and show through a variety of numerical experiments that it preserves the qualitative dynamics of neuronal spiking, in particular, \textit{amplitudes, frequencies} and \textit{phases} of the underlying oscillations even for large time steps $\Delta$. 

The article is organised as follows. In Section \ref{sec:2_FHN}, we introduce necessary mathematical preliminaries and notations, and we discuss equations of interest and relevant properties. In Section~\ref{sec:3_FHN}, we present the proposed splitting method. In Section \ref{sec:4_FHN}, we establish its  mean-square convergence. In Section \ref{sec:5_FHN}, we prove its $1$-step hypoellipticity, establish its geometric ergodicity, derive an (asymptotic) second moment bound and illustrate these results on a one-dimensional cubic model problem. 
In Section \ref{sec:6_FHN}, we apply the proposed splitting approach to the stochastic FHN model \eqref{FHN}. In Section \ref{sec:7_FHN}, we provide a variety of numerical experiments, illustrating the theoretical results and reporting comparisons with different tamed/truncated variants of the Euler-Maruyama method. Conclusions are given in Section \ref{sec:8_FHN}.

\section{Model and properties}
\label{sec:2_FHN}

Throughout, the following notations are used.
\begin{notation}
	Let $x,y \in \mathbb{R}^d$ be two generic vectors. Then $x_l$ denotes the $l$-th entry of $x$, $x^\top$ the transpose of $x$, $\norm{x}=(x_1^2+\ldots+ x_d^2)^{1/2}$ the Euclidean norm of $x$ and $(x,y)~=~x_1y_1~+~\ldots~+~x_dy_d$ the scalar product of $x$ and $y$. Further, let $A,B \in \mathbb{R}^{d\times d}$ be two generic matrices. Then $a_{lj}$ denotes the component in the $l$-th row and $j$-th column of $A$, $A^{\top}$ the transpose of $A$,  
	$0_d$ the $d$-dimensional zero vector and $\mathbb{I}_{d}$ the $d \times d$-dimensional identity matrix. 
	Moreover, we denote by $\norm{A}=\sqrt{\lambda_{\textrm{max}}(A^\top A)}$ the matrix norm which is induced by the Euclidean norm, where $\lambda_{\textrm{max}}(A)$ is the largest eigenvalue of $A$, and with $\mu(A)=\lambda_{\textrm{max}}((A+A^\top)/2)$ the real-valued logarithmic norm which results from the Euclidean norm and its induced matrix norm.
\end{notation}

Let $(\Omega,\mathcal{F},\mathbb{P})$ be a complete probability space with filtration $(\mathcal{F}(t))_{t \in [0,T]}$. Further, let $(W(t))_{t \in [0,T]}$ be a $m$-dimensional Wiener process defined on that space and adapted to $(\mathcal{F}(t))_{t \in [0,T]}$. We consider the $d$-dimensional autonomous SDE of additive noise type \eqref{eq:semi_linear_SDE}
\begin{equation*}
	dX(t)=F(X(t))dt + \Sigma dW(t):=\bigl[ AX(t) + N(X(t)) \bigr]dt + \Sigma dW(t), \quad X(0)=X_0, % \quad t \in [0,T],
\end{equation*}
where $t \in [0,T]$, $T>0$, $A\in\mathbb{R}^{d\times d}$, $\Sigma \in \mathbb{R}^{d \times m}$, $F: \mathbb{R}^d \to \mathbb{R}^d$ and $N: \mathbb{R}^d \to \mathbb{R}^d$ are locally Lipschitz continuous. The initial value $X_0$ is an $\mathcal{F}(0)$-measurable $\mathbb{R}^d$-valued random variable which is independent of $(W(t))_{t \in [0,T]}$ and such that $\mathbb{E}\left[ \norm{X_0}^{2p} \right] <\infty$ for all $p \geq 1$.

Conditions required to ensure the existence of a unique strong solution of SDE~\eqref{eq:semi_linear_SDE}, which is regular in the sense of \cite{Khasminskii2012}, i.e., it is defined on the entire interval $[0,T]$ such that sample paths do not blow up to infinity in finite time, are, e.g., discussed in \cite{Alyushina1988,Khasminskii2012,Krylov1991,Mao2011}. Here, we follow the setting in \cite{Hutzenthaler2012_2,Kelly2017,Tretyakov2012} and suppose that the drift satisfies a global one-sided Lipschitz condition and is allowed to grow polynomially at infinity. 
It suffices to place these conditions on $N$:
\begin{assumption}\label{assum:1}
	\begin{enumerate}[label=\text{(A\arabic*)}]
		\item \label{(A1)} The function $N$ is globally one-sided Lipschitz continuous, i.e., there exists a constant $c_1>0$ such that 
		\begin{equation*}
			(x-y,N(x)-N(y))\leq c_1\norm{x-y}^2, \quad \forall \ x,y\in \mathbb{R}^d.
		\end{equation*}
		\item \label{(A2)} The function $N$ grows at most polynomially, i.e., there exist constants $c_2>0$ and $\chi\geq 1$ such that 
		\begin{equation*}
			\norm{N(x)-N(y)}^2\leq c_2(1+\norm{x}^{2\chi-2}+\norm{y}^{2\chi-2})||x-y||^2, \quad \forall \ x,y\in\mathbb{R}^d.
		\end{equation*}
	\end{enumerate}
\end{assumption}\noindent
Assumption \ref{assum:1} also ensures the finiteness of the moments of the solution of \eqref{eq:semi_linear_SDE} \cite{Higham2002,Khasminskii2012,Tretyakov2012,Zhang2017}. In particular, there exists a constant $K(T,p)>0$ such that
\begin{equation}\label{eq:bm}
	\mathbb{E}\left[ \sup_{0\leq t \leq T} \norm{X(t)}^{2p} \right] \leq K(T,p)\left(1+\mathbb{E}\left[ \norm{X_0}^{2p}\right]\right).
\end{equation}

Moreover, the process $(X(t))_{t \in [0,T]}$ is a Markov process. Denoting $\mathcal{B}(\mathbb{R}^d)$ the Borel sigma-algebra on $\mathbb{R}^d$, its transition probability is defined as
\begin{equation}\label{eq:trans_prob}
	P_{t}(\mathcal{A},x):=\mathbb{P}\left( X(t) \in \mathcal{A} | X(0)=x \right),
\end{equation}
where $\mathcal{A} \in \mathcal{B}(\mathbb{R}^d)$. This corresponds to the probability that the process reaches a Borel set $\mathcal{A} \subset \mathbb{R}^d$ at time $t$, provided that it started in $x \in \mathbb{R}^d$ at time $0< t$.

\subsection{Noise structure: ellipticity and hypoellipticity}
\label{sec:2:1}

Depending on the noise structure, two classes of models are obtained. The first class is called \textit{elliptic} and corresponds to SDEs with a non-degenerate diffusion matrix, i.e., $\Sigma\Sigma^\top$ is of full rank. 
In particular, we consider the case $d=m$ and a diagonal matrix $\Sigma=\textrm{diag}[\sigma_1,\ldots,\sigma_{d}]$ with entries $\sigma_j>0$ for $j=1,\ldots,d$.

The second class corresponds to SDEs with degenerate diffusion matrix, as it naturally occurs in many application models. Following the notion in \cite{Ditlevsen2019}, we consider $m=d-1$ and $\Sigma$ given by
\begin{equation}\label{eq:hypo_SDE}
	\Sigma:=
	\begin{pmatrix}
		0_{d-1}^{\top} \\
		\Gamma
	\end{pmatrix},
\end{equation}
where $\Gamma=\textrm{diag}[\sigma_1,\ldots,\sigma_{d-1}] \in \mathbb{R}^{(d-1)\times (d-1)}$ is a diagonal matrix with entries $\sigma_j>0$ for $\text{$j=1,\ldots,d-1$}$. The first component of the solution $(X(t))_{t \in [0,T]}$ is called \textit{smooth}, since it is not directly affected by the noise. The remaining $d-1$ components are called \textit{rough}, because the noise acts directly on them.
In this scenario, SDE \eqref{eq:semi_linear_SDE} is often \textit{hypoelliptic}.
It means that the transition probability \eqref{eq:trans_prob} admits a smooth density, even though $\Sigma\Sigma^\top$ is not of full rank. This is the case when the SDE satisfies the weak H\"ormander condition, based on the concept of Lie-brackets~\cite{Nualart1995}.
In \cite{Ditlevsen2019}, it was shown that a necessary and sufficient condition for the SDE to meet the weak H\"ormander condition is that at least one of the rough coordinates of the process $(X(t))_{t \in [0,T]}$ appears in the first component $F_1(X(t))$ of the drift, that is
\begin{equation}\label{eq:cond_hypo}
	\forall x \in \mathbb{R}^d, \ \left(\partial_r F_1(x), \sigma^j\right) \neq 0 \quad \text{for at least one} \ j=1,\ldots,d-1,
\end{equation}
where $\sigma^j$ denotes the $j$-th column vector of $\Gamma$ and $\partial_r F_1(x):=(\partial_{x_2} F_1(x),\ldots, \partial_{x_{d}} F_1(x) )^\top$ is the vector of partial derivatives of the first entry of the drift with respect to the rough components. This setting can be extended to multiple smooth coordinates by requiring that at least one of the rough coordinates enters those components of the drift, where noise does not directly act upon.

\subsection{Lyapunov structure: geometric ergodicity}
\label{sec:2:2}

Here, a particular interest lies in SDEs of type \eqref{eq:semi_linear_SDE} where the drift $F(X(t))$ satisfies the following dissipativity condition
\begin{equation}\label{eq:dissipative}
	(F(x),x)\leq \alpha -\delta \norm{x}^2, \quad \forall \ x\in \mathbb{R}^d,
\end{equation}
where $\alpha,\delta>0$. Condition \eqref{eq:dissipative} ensures that the function $L:\mathbb{R}^d\to [1,\infty)$ defined by $\text{$L(x):=1+\norm{x}^2$}$ is a Lyapunov function for \eqref{eq:semi_linear_SDE}, see \cite{Mattingly2002}. That is $L(x)\to\infty$ as $\norm{x}\to \infty$, and there exist constants $\rho,\eta>0$ such that
\begin{equation}\label{eq:Lyapunov}
	\mathcal{L}\{ L(x) \} \leq -\rho L(x) +\eta,
\end{equation}
where $\mathcal{L}$ is the generator of the SDE given by
\begin{equation*}
	\mathcal{L}\{g(x)\}=\sum\limits_{l=1}^{d} F_l(x) \frac{\partial g}{\partial x_l}(x)+\frac{1}{2}\sum\limits_{l,j=1}^{d}\left[ \Sigma\Sigma^\top \right]_{lj} \frac{\partial^2 g}{\partial x_l\partial x_j}(x),
\end{equation*}
for sufficiently smooth functions $g:\mathbb{R}^d\to\mathbb{R}$. 
The existence of a Lyapunov function satisfying~\eqref{eq:Lyapunov} is the key to establish the \textit{geometric ergodicity} of the solution of \eqref{eq:semi_linear_SDE}. This property means that the distribution of the Markov process $(X(t))_{t\in [0,T]}$ converges exponentially fast to a unique invariant distribution~$\pi$, satisfying
\begin{equation*}
	\pi(\mathcal{A})=\int\limits_{\mathbb{R}^d}P_t(\mathcal{A},x)\pi(dx), \quad \forall \ \mathcal{A} \in \mathcal{B}(\mathbb{R}^d), \ t \in [0,T].
\end{equation*}
In particular, if SDE \eqref{eq:semi_linear_SDE} is elliptic, the existence of a Lyapunov function meeting Condition~\eqref{eq:Lyapunov} suffices to establish the geometric ergodicity of $(X(t))_{t\in [0,T]}$. If SDE \eqref{eq:semi_linear_SDE} is not elliptic, the process is geometrically ergodic, if, in addition to fulfilling Condition~\eqref{eq:Lyapunov}, it is hypoelliptic and satisfies the irreducibility condition $P_t(\mathcal{A},x)>0$ for all open sets $\mathcal{A}\in \mathbb{B}(\mathbb{R}^d)$ and $x\in~\mathbb{R}^d$. The reader is referred to \cite{Mattingly2002} and the references therein for further details.

\section{Splitting method}
\label{sec:3_FHN}

Consider a discretised time interval $[0,T]$ with equidistant time steps $\Delta=t_{i}-t_{i-1} \in (0,\Delta_0]$, $\Delta_0~\in~(0,1)$, $i=1,\dots,n$, where $t_0=0$ and $t_n=T$. Throughout, we denote by
$(\widetilde{X}(t_i))_{i=0,\ldots,n}$ a numerical solution of SDE \eqref{eq:semi_linear_SDE}, 
approximating the process $(X(t))_{t \in [0,T]}$  at $t_i$, where $\widetilde{X}(0):=X_0$. 

A numerical splitting solution is obtained based on the following three steps \cite{Blanes2009,Mclachlan2002}: 
\begin{itemize}
	\item[(i)] Split the equation of interest into exactly solvable subequations, which may consist of deterministic and/or stochastic dynamical systems;
	\item[(ii)] Derive the exact solutions of these subequations; 
	\item[(iii)] Compose the derived solutions in a proper way.
\end{itemize}
In this section, the following splitting strategy is proposed for SDEs of type \eqref{eq:semi_linear_SDE}.

\paragraph*{Step (i): Choice of the subequations}
To make use of the treatable underlying stochastic linear dynamics, we propose to split Equation \eqref{eq:semi_linear_SDE} into the following two subequations
\begin{eqnarray}
	d X^{[1]}(t)&=&AX^{[1]}(t) dt + \Sigma dW(t), \quad X^{[1]}(0)=X^{[1]}_0, \quad t \in [0,T],
	\label{SDE_FHN} \\
	dX^{[2]}(t)&=&N(X^{[2]}(t))dt, \quad X^{[2]}(0)=X^{[2]}_0, \quad t \in [0,T].
	\label{ODE_FHN}
\end{eqnarray}
This splitting strategy is an extension of the method presented in \cite{Ableidinger2017}, where the authors consider a globally Lipschitz Hamiltonian type equation with uniformly bounded non-linear terms. Our method considers a more general class of coefficients $N(X(t))$, including functions which are allowed to grow polynomially at infinity according to Assumption \ref{assum:1}.

\paragraph*{Step (ii): Exact solution of the subequations}
In the following, we discuss the subequations \eqref{SDE_FHN} and \eqref{ODE_FHN}, and denote by $\varphi_t^{[k]}(X_0)$, $k=1,2$, their exact solutions (flows) at time $t$ and starting from $X_0$. The first subequation \eqref{SDE_FHN} is a linear SDE. It can be solved exactly, even when the dimension $d$ is large and independent of whether the equation has an elliptic or hypoelliptic noise structure \cite{Arnold1974,Mao2011}. In particular, the exact solution of \eqref{SDE_FHN} is given by
\begin{equation}\label{eq:SDE}
	X^{[1]}(t)=e^{At}X_0^{[1]} + \int\limits_{0}^{t} e^{A(t-s)} \Sigma \ dW(s).
\end{equation}
The It\^{o} integral in \eqref{eq:SDE} is normally distributed with mean $0_d$. Moreover, using It\^{o}'s isometry and the fact that the components of the Wiener process are independent, its $d \times d$-dimensional covariance matrix is given by
\begin{equation}\label{eq:Cov}
	C(t)=\int\limits_{0}^{t} e^{A(t-s)}\Sigma\Sigma^{\top}(e^{A(t-s)})^{\top} \ ds.
\end{equation}
Hence, paths of \eqref{SDE_FHN} can be simulated exactly at the discrete time points $t_i$. In particular,
\begin{equation}
	\varphi_{\Delta}^{[1]}(X^{[1]}(t_{i-1})):=X^{[1]}(t_{i})=e^{A\Delta}  X^{[1]}(t_{i-1}) + \xi_{i-1}, \quad i=1,\dots,n,
	\label{Exact_SDE_FHN}
\end{equation}
where the $\xi_{i-1}$ are independent and identically distributed $d$-dimensional Gaussian vectors with mean $0_d$ and covariance matrix $C(\Delta)$ given by \eqref{eq:Cov}. 

Due to \ref{(A1)} of Assumption \ref{assum:1}, the global solution of the second subequation \eqref{ODE_FHN} exists, i.e., it is defined on the entire interval $[0,T]$ such that it does not blow up to infinity in finite time \cite{Humphries2001}. At the discrete time points $t_i$, we have
\begin{equation}
	\varphi_{\Delta}^{[2]}(X^{[2]}(t_{i-1})):=X^{[2]}(t_i)=f(X^{[2]}(t_{i-1});\Delta), \quad i=1,\dots,n,
	\label{Exact_ODE_FHN}
\end{equation}
where $f:\mathbb{R}^d \to \mathbb{R}^d$ denotes the exact solution of Equation \eqref{ODE_FHN}. 

\begin{remark}
	To establish the boundedness of the moments (Lemma \ref{Lemma3}), the Lyapunov condition (Theorem \ref{thm:Lyapunov_discrete}) and the asymptotic second moment bound (Corollary \ref{cor:asym_2nd_moment} and \ref{cor:bounds_1dim}), we exploit properties of the exact solution $f$ of Equation~\eqref{ODE_FHN}. These will be illustrated on a cubic model problem and the FHN model \eqref{FHN} in Section \ref{sec:5_FHN} and Section \ref{sec:6_FHN}, respectively. However, note that some of the results presented in the following are formulated with conditions not involving f directly (see Lemma \ref{Lemma2} on the mean-square consistency and Theorem \ref{thm:hypo_N_LT} on the $1$-step hypoellipticity). Therefore, these results would also hold when a numerical method to approximate the solution of Equation \eqref{ODE_FHN} is used. We refer to \cite{Hairer2000,Humphries2001} for an exhaustive discussion of numerical methods for locally Lipschitz ODEs. 
\end{remark}

\paragraph*{Step (iii): Composition of the exact solutions}
To finally obtain a numerical solution of SDE \eqref{eq:semi_linear_SDE}, the exact solutions \eqref{Exact_SDE_FHN} and \eqref{Exact_ODE_FHN} of the subequations \eqref{SDE_FHN} and \eqref{ODE_FHN} are composed in every iteration step. In particular, we investigate the following explicit method
\begin{eqnarray}\label{SP1_FHN}
	\widetilde{X}^{\textrm{LT}}(t_i)&=&\left( \varphi_\Delta^{[1]} \circ \varphi_\Delta^{[2]} \right) \bigl(\widetilde{X}^{\textrm{LT}}(t_{i-1}) \bigr)=e^{A\Delta}f(\widetilde{X}^{\textrm{LT}}(t_{i-1});\Delta)+\xi_{i-1}, 
\end{eqnarray}
which is based on the Lie-Trotter (LT) composition approach \cite{Trotter1959}. 

Note that the matrix exponential $e^{A\Delta}$ and the covariance matrix $C(\Delta)$ required in \eqref{SP1_FHN} 
have to be precomputed only once for a given time step $\Delta$, and the normal random variables $\xi_{i-1}$, $i=1,\ldots,n$, can be obtained using a Cholesky decomposition of the covariance matrix~$C(\Delta)$. 

\section{Mean-square convergence}
\label{sec:4_FHN}

In this section, mean-square convergence of order $1$ is proved for the constructed splitting method. It has been observed in the globally Lipschitz case that splitting methods have the same convergence order as the standard Euler-Maruyama method, i.e., order $1$ in the case of additive noise, see, e.g., \cite{Ableidinger2017,Alamo2016,Milstein2003}. We extend this result to the one-sided Lipschitz case. 

Throughout this section, $K$ denotes a generic constant, which may depend on $T$, $p$ and $\Delta_0$, but is independent of $\Delta$ and $i$.

\subsection{Required background}

To establish mean-square convergence, we rely on Theorem 2.1 of \cite{Tretyakov2012}, which provides an extension of Milstein's fundamental theorem on the mean-square order of convergence for globally Lipschitz coefficients~\cite{Milstein1988} (see also Theorem 1.1 in \cite{Milstein2004}) to the considered setting specified in Assumption~\ref{assum:1}. To facilitate the illustration of our results, we recall this statement in Theorem~\ref{thm:Zang} below, after defining the required ingredients of mean-square consistency and boundedness of moments.

Let $X_{t_{i-1},x}(t_i)$ denote the true solution at time $t_i$ starting from $x$ at time $t_{i-1}$, i.e., $X(t_{i-1})=x$, and $\widetilde{X}_{t_{i-1},x}(t_i)$ the one-step approximation used to construct a numerical solution $\widetilde{X}(t_i)$. In particular, the one-step approximation of the numerical method discussed in the previous section is defined by \eqref{SP1_FHN}, where $\widetilde{X}^{\textrm{LT}}(t_{i-1})$ is replaced by $x$.
\begin{definition}\label{eq:ms_consistency}
	The one-step approximation $\widetilde{X}_{t_{i-1},x}(t_i)$ of a numerical solution $\widetilde{X}(t_i)$ of SDE \eqref{eq:semi_linear_SDE} is mean-square consistent of order $q_2-1/2$, if for some $p\geq 1$, there exist $\alpha\geq 1$, $\Delta_0>0$ and $K>0$ such that for arbitrary $t_i$, $i=1,\ldots,n$, $x \in \mathbb{R}^d$, and all $\Delta\in(0,\Delta_0]$, it holds that
	\begin{eqnarray*}\hspace{-0.5cm}
		\nonumber\norm{\mathbb{E}\left[X_{t_{i-1},x}(t_i)-\widetilde X_{t_{i-1},x}(t_i)\right]} & \leq  K\left( 1+\norm{x}^{2\alpha} \right)^{1/2} \Delta^{q_1}, \\ \left(\mathbb{E}\left[\norm{X_{t_{i-1},x}(t_i)-\widetilde X_{t_{i-1},x}(t_i)}^{2p}\right]\right)^{1/(2p)}&\leq  K\left( 1+\norm{x}^{2\alpha p} \right)^{1/(2p)} \Delta^{q_2},
	\end{eqnarray*}
	with $q_2 \geq 1/2$ and $q_1\geq q_2+1/2$.
\end{definition}\noindent
Besides mean-square consistency, the boundedness of the moments of the numerical solution has to be proved. In the globally Lipschitz case, this is guaranteed by the linear growth bounds of the coefficients.
\begin{definition}\label{eq:bm_num}
	A numerical solution $\widetilde{X}(t_i)$ of SDE \eqref{eq:semi_linear_SDE} has bounded moments, if for any $p \geq 1$, there exist $\Delta_0>0$ and $K>0$ such that for all $\Delta\in(0,\Delta_0]$ and $i=0,\ldots,n$, it holds that
	\begin{equation*}
		\mathbb{E}\left[ \norm{\widetilde X(t_i)}^{2p}\right] \leq K\left(1+ \mathbb{E}\left[\norm{X_0}^{2p}\right]\right).
	\end{equation*}
\end{definition}\noindent
Based on the above defined ingredients, the following theorem guarantees mean-square convergence.
\begin{theorem}[Theorem 2.1 in Tretyakov and Zhang (2013) \cite{Tretyakov2012}]\label{thm:Zang} 
	Let $\widetilde{X}(t_i)$ denote a numerical solution of SDE \eqref{eq:semi_linear_SDE} at time $t_i$ starting at $X_0$, constructed using the one-step approximation $\widetilde{X}_{t_{i-1},x}(t_i)$. Further, let Assumption \ref{assum:1} hold. If 
	\begin{itemize} 
		\item[(i)] The one-step approximation $\widetilde{X}_{t_{i-1},x}(t_i)$ is mean-square consistent of order $q_2-1/2$ in the sense of Definition \ref{eq:ms_consistency}.
		\item[(ii)] The numerical method $\widetilde{X}(t_i)$ has bounded moments in the sense of Definition \ref{eq:bm_num}.
	\end{itemize}
	Then the numerical method $\widetilde{X}(t_i)$ is mean-square convergent of order $q_2-1/2$, i.e., for any $n$ and $i=0,\ldots,n$, the following inequality holds:
	\begin{equation*}
		\left( \mathbb{E}\left[ \norm{ X(t_i)-\widetilde{X}(t_i) }^{2p} \right]\right)^{1/(2p)} \leq K \left( 1+ \mathbb{E}\left[ \norm{X_0}^{2pc} \right] \right)^{1/(2p)}\Delta^{q_2-1/2},
	\end{equation*}
	where $K>0$ and $c\geq 1$.
\end{theorem}\noindent

\vspace{-0.2cm}
\subsection{Mean-square convergence of the splitting method}

In the following, we prove the required Conditions $(i)$ and $(ii)$ of Theorem \ref{thm:Zang} for the constructed splitting method.

Condition $(i)$ can be proved in a similar fashion as Lemma 2.1 in \cite{Milstein2003} (globally Lipschitz case) and Lemma 3.2 in \cite{Tretyakov2012,Zhang2017} (locally Lipschitz case). These proofs rely on the mean-square consistency of the Euler-Maruyama method, which is given by
\begin{equation}\label{EM_FHN}
	\widetilde{X}^{\textrm{EM}}(t_i)=\widetilde{X}^{\textrm{EM}}(t_{i-1})
	+ F(\widetilde{X}^{\textrm{EM}}(t_{i-1}))\Delta + \Sigma \sqrt{\Delta} \psi_{i-1},
\end{equation}
where the $\psi_{i-1} \sim \mathcal{N}(0_m,\mathbb{I}_m)$, $i=1,\ldots,n$, are independent and identically distributed $m$-dimensional standard Gaussian vectors \cite{Kloeden1992,Milstein2004}.
\begin{lemma}[Mean-square consistency]\label{Lemma2} 
	Let $\widetilde{X}^{\textrm{LT}}_{t_{i-1},x}(t_i)$ be the one-step approximation of the splitting method defined through \eqref{SP1_FHN} and let Assumption \ref{assum:1} hold. Further, assume that the drift $F(x)$ has continuous first and second order derivatives in $x$ which satisfy a polynomial growth condition of the form \ref{(A2)}. Then $\widetilde{X}^{\textrm{LT}}_{t_{i-1},x}(t_i)$ is mean-square consistent of order $1$ in the sense of Definition~\ref{eq:ms_consistency}. In particular, $q_1=2$ and $q_2=3/2$.
\end{lemma}\noindent

\begin{proof}
	Consider first the one-step approximation 
	\begin{equation}\label{eq:EM_one_step}
		\widetilde{X}^{\textrm{EM}}_{t_{i-1},x}(t_i)=x+F(x)\Delta+\Sigma\sqrt{\Delta}\psi_{i-1}
	\end{equation}
	of the Euler-Maruyama method \eqref{EM_FHN} applied to SDE \eqref{eq:semi_linear_SDE}. 
	Since we consider the case of additive noise and \ref{(A2)} holds for the drift $F(x)=Ax+N(x)$ by assumption of Lemma \ref{Lemma2}, it follows from the proof of Lemma 3.2 in \cite{Zhang2017} that \eqref{eq:EM_one_step} 
	satisfies Definition \ref{eq:ms_consistency} with $q_1=2$, $q_2=3/2$. Thus, it suffices to compare the splitting and Euler-Maruyama methods.
	
	Since the drift of the stochastic subequation \eqref{SDE_FHN} of the splitting framework grows linearly, using again Lemma 3.2 in \cite{Zhang2017} (see also Section 1.1.5 in \cite{Milstein2004}), its solution can be expressed as 
	\begin{equation}\label{SDE_remainder}
		\varphi_{\Delta}^{[1]}(x)=x+ Ax \Delta + \Sigma \sqrt{\Delta} \psi_{i-1}+r_s(x,\Delta),
	\end{equation} 
	where $r_s(x,\Delta)$ satisfies the inequalities of Definition \ref{eq:ms_consistency} for $\alpha=1$, $q_1=2$ and $q_2=3/2$. In particular, we have that	
	\begin{equation}\label{eq:rs}
		\begin{aligned}
			\norm{\mathbb{E}\left[ r_s(x,\Delta) \right]} &\leq  K\left( 1+\norm{x}^{2} \right)^{1/2} \Delta^{2}, \\
			\left(\mathbb{E}\left[\norm{ r_s(x,\Delta) }^{2p}\right]\right)^{1/(2p)} &\leq  K\left( 1+\norm{x}^{2 p} \right)^{1/(2p)} \Delta^{3/2}. 
		\end{aligned}
	\end{equation}	
	Similarly, since $N(x)$ satisfies \ref{(A2)}, the solution of the deterministic subequation~\eqref{ODE_FHN} of the splitting framework can be expressed as 
	\begin{equation}\label{ODE_remainder_2}
		\varphi_{\Delta}^{[2]}(x)=x+ N(x)\Delta+r_d(x,\Delta),
	\end{equation}
	where 
	\begin{equation}\label{eq:rd}
		\norm{ r_d(x,\Delta) } \leq  K\left( 1+\norm{x}^{2\alpha} \right)^{1/2} \Delta^{2},
	\end{equation}
	for some $\alpha \geq 1$.
	The one-step approximation of the Lie-Trotter splitting method is then obtained by composing the above expressions \eqref{SDE_remainder} and \eqref{ODE_remainder_2}, yielding
	\begin{eqnarray*}
		\widetilde{X}^{\textrm{LT}}_{t_{i-1},x}(t_i)&=&(\varphi_{\Delta}^{[1]} \ \circ \ \varphi_{\Delta}^{[2]})(x)=x+N(x)\Delta + r_d(x,\Delta)  + Ax\Delta + AN(x)\Delta^2+Ar_d(x,\Delta)\Delta \\&& \hspace{0.5cm} +\Sigma\sqrt{\Delta}\psi_{i-1}+r_s\bigl(x+N(x)+r_d(x,\Delta),\Delta\bigr).
	\end{eqnarray*}
	Thus, the difference between the splitting and Euler-Maruyama methods becomes
	\begin{eqnarray*}
		r^{\textrm{LT}}(x,\Delta)&:=&\widetilde{X}^{\textrm{LT}}_{t_{i-1},x}(t_i)-\widetilde{X}^{\textrm{EM}}_{t_{i-1},x}(t_i)\\ &=&r_d(x,\Delta)+AN(x)\Delta^2+A r_d(x,\Delta) \Delta+r_s\bigl(x+N(x)+r_d(x,\Delta),\Delta\bigr).
	\end{eqnarray*}
	Using \ref{(A2)} and the inequalities \eqref{eq:rs} and \eqref{eq:rd}, it follows that $r^{\textrm{LT}}(x,\Delta)$ also satisfies the inequalities in Definition \ref{eq:ms_consistency} for $q_1=2$, $q_2=3/2$ and some $\alpha\geq 1$. This concludes the proof.
\end{proof}

Now, we establish the boundedness of the moments of the splitting method. Intuitively, this is guaranteed by the use of the global exact solution of the locally Lipschitz ODE \eqref{ODE_FHN}, which is defined on the entire interval $[0,T]$ without any explosion occuring in finite time. Thus, the iterative composition of this function with the solution of the linear SDE via the Lie-Trotter method \eqref{SP1_FHN} does not cause an explosion of the moments in finite time either. The formal proof of this result, provided in Lemma \ref{Lemma3}, is done in the spirit of the proof of Proposition~3 in \cite{Brehier2018}. 

\begin{lemma}[Boundedness of moments]\label{Lemma3} 
	Let $\widetilde{X}^{\textrm{LT}}(t_i)$ be the splitting method defined through \eqref{SP1_FHN} and let Assumption \ref{assum:1} hold.
	Then $\widetilde{X}^{\textrm{LT}}(t_i)$ is mean-square bounded in the sense of Definition~\ref{eq:bm_num}. 
\end{lemma}\noindent

\begin{proof}
	Consider the linear SDE
	\begin{equation*}
		dZ(t)=AZ(t)dt+\Sigma dW(t), \quad Z(0)=Z_0=0_d.
	\end{equation*}
	Its exact solution is given by
	\begin{equation*}
		Z(t)=\int\limits_{0}^{t} e^{A(t-s)} \Sigma dW(s),
	\end{equation*}
	where $Z(t)$ is normally distributed with mean vector $0_d$ and covariance matrix $C(t)$ as defined in \eqref{eq:Cov}. Consequently, the moments of $Z(t)$ are bounded, i.e., for any $p\geq 1$ there exists $K_Z(T,p)>0$ such that
	\begin{equation}\label{eq:moments_Z}
		\mathbb{E}\left[ \sup\limits_{0\leq t\leq T} \norm{Z(t)}^{2p} \right] \leq K_Z(T,p).
	\end{equation}
	
	Now, define the process $R(t_i):=\widetilde{X}^{\textrm{LT}}(t_i)-Z(t_i)$. It suffices to prove the boundedness of the moments of $R(t_i)$. Note that in a discretised regime we have that $Z(t_i)=e^{A\Delta}Z(t_{i-1})+\xi_{i-1}$. Thus, 
	\begin{eqnarray*}
		\norm{R(t_i)}&=&\norm{ e^{A\Delta} \left( f(\widetilde{X}^{\textrm{LT}}(t_{i-1});\Delta) - Z(t_{i-1}) \right) } \\
		&=& \norm{ e^{A\Delta}  \Bigl( f(R(t_{i-1})+Z(t_{i-1});\Delta) -f(Z(t_{i-1});\Delta) + f(Z(t_{i-1});\Delta)  -Z(t_{i-1})\Bigr) }.
	\end{eqnarray*}
	Using that $\norm{e^{A\Delta}x}\leq\norm{e^{A\Delta}}\norm{x}\leq e^{\mu(A)\Delta}\norm{x}$ for all $x \in \mathbb{R}^d$, we obtain
	\begin{eqnarray*}
		\norm{R(t_i)}&\leq& 
		e^{\mu(A)\Delta} \norm{  f(R(t_{i-1})+Z(t_{i-1});\Delta) -f(Z(t_{i-1});\Delta) } \\ && \hspace{0.5cm} + e^{\mu(A)\Delta} \norm{  f(Z(t_{i-1});\Delta)  -Z(t_{i-1}) }.
	\end{eqnarray*}	
	
	Since the function $N:\mathbb{R}^d\to\mathbb{R}^d$ satisfies \ref{(A1)}, using the continuous Gronwall Lemma, the function $f:\mathbb{R}^d\to\mathbb{R}^d$ fulfils the following global Lipschitz condition
	\begin{equation*}
		\norm{f(x;\Delta)-f(y;\Delta)} \leq e^{c_1 \Delta}\norm{x-y}, \quad \forall \ x,y \in \mathbb{R}^d,
	\end{equation*}
	where the constant $c_1$ is the same as in Assumption \ref{(A1)}, see, e.g., Theorem 1.2.17 in \cite{Humphries2001}.	
	Moreover, using the Taylor series expansion \eqref{ODE_remainder_2} $f(x;\Delta)=x+N(x)\Delta+r_d(x,\Delta)$, and applying \ref{(A2)} and \eqref{eq:rd}, we obtain
	\begin{equation*}
		\norm{f(x;\Delta)-x} = \norm{ \Delta  N(x) + r_d(x,\Delta) } \leq \Delta \norm{N(x)} + \norm{r_d(x,\Delta)} \leq  \bar{c} (1+\norm{x}^{\hat{c}}) \Delta,
	\end{equation*}
	where $\bar{c}$ and $\hat{c}$ are positive constants. Thus, defining $\tilde{c}:= \max\{|\mu(A)|,c_1\}>0$, we get that
	\begin{eqnarray*}
		\norm{R(t_i)} &\leq& 
		e^{\tilde{c}\Delta} \norm{ R(t_{i-1})} +  e^{\tilde{c}\Delta}  \Delta \bar{c}  \bigl( 1+\norm{ Z(t_{i-1})}^{\hat{c}} \bigr).
	\end{eqnarray*}\noindent
	Now, we can perform back iteration, obtaining
	\begin{eqnarray*}
		\norm{R(t_i)} &\leq& e^{\tilde{c}t_i} \norm{ R_0} + \bar{c} \Delta \sum_{k=1}^{i}e^{\tilde{c}k\Delta} \Bigl( 1+\norm{ Z(t_{i-k}) }^{\hat{c}} \Bigr) \\
		&\leq& e^{\tilde{c}T} \norm{ X_0} + \bar{c}  \Bigl( 1+ \sup_{0\leq l \leq i-1}  \norm{ Z(t_{l}) }^{\hat{c}} \Bigr)   \Delta \sum_{k=1}^{i}e^{\tilde{c}k\Delta}, 
	\end{eqnarray*}
	where we used that $R_0=X_0$, since $Z_0=0_d$. Using that
	\begin{equation*}\label{eq:geom1}
		\Delta\sum_{k=1}^{i}e^{\tilde{c} k\Delta}=(e^{\tilde{c} t_i}-1)\frac{\Delta e^{\tilde{c}\Delta}}{e^{\tilde{c}\Delta}-1} \leq (e^{\tilde{c}T}-1)\frac{\Delta_0e^{\tilde{c}\Delta_0}}{e^{\tilde{c}\Delta_0}-1}, \quad \forall \ \Delta \in (0,\Delta_0],
	\end{equation*}
	we get that
	\begin{eqnarray*}
		\norm{R(t_i)} &\leq& e^{\tilde{c}T} \norm{ X_0} + \bar{c} (e^{\tilde{c}T}-1) \frac{\Delta_0e^{\tilde{c}\Delta_0}}{e^{\tilde{c}\Delta_0}-1} \Bigl( 1+\sup_{0\leq l \leq i-1}   \norm{ Z(t_{l}) }^{\hat{c}}  \Bigr). 
	\end{eqnarray*}
	Thus, there exists a constant ${K}(T,\Delta_0)>0$ such that
	\begin{equation*}
		\norm{R(t_i)}\leq {K}(T,\Delta_0)\Bigl( 1+\norm{X_0} + \sup_{0\leq l \leq i-1}  \norm{ Z(t_{l}) }^{\hat{c}}   \Bigr).
	\end{equation*}
	
	Considering the $2p$-th moments and using \eqref{eq:moments_Z} concludes the proof.
\end{proof}

Based on the above results, we establish the mean-square convergence of the splitting method in the following theorem.
\begin{theorem}[Mean-square convergence]\label{thm:convergence_FHN}
	Let $\widetilde{X}^{\textrm{LT}}(t_i)$ be the splitting method defined through \eqref{SP1_FHN} and let the assumptions of Theorem \ref{thm:Zang}, Lemma \ref{Lemma2} and Lemma \ref{Lemma3} hold. Then $\widetilde{X}^{\textrm{LT}}(t_i)$ is mean-square convergent of order $1$.
\end{theorem}
\begin{proof}
	The result is a direct consequence of Theorem \ref{thm:Zang}, Lemma \ref{Lemma2} and Lemma \ref{Lemma3}.
\end{proof}\noindent

Note that, in contrast to ODEs \cite{Hairer2006}, the mean-square convergence order of splitting methods for SDEs cannot be increased by using compositions based on fractional steps. Indeed, to achieve this in the stochastic scenario, higher-order stochastic integrals would be required~\cite{Milstein2003}. Thus, the splitting method 
\begin{eqnarray} 
	\label{SP3_FHN}
	\widetilde{X}^{\textrm{S}}(t_i)&=&\left( \varphi_{\Delta/2}^{[2]} \circ \varphi_{\Delta}^{[1]} \circ \varphi_{\Delta/2}^{[2]} \right)\bigl(\widetilde{X}^{\textrm{S}}(t_{i-1}) \bigr)=f\left(e^{A\Delta}f\bigl(\widetilde{X}^{\textrm{S}}(t_{i-1});\Delta/2\bigr) + \xi_{i-1};\Delta/2 \right),
\end{eqnarray}
which is based on the Strang (S) composition approach \cite{Strang1968}, is expected to also have mean-square order~$1$. Nevertheless, it has been observed that Strang methods may perform better than Lie-Trotter methods in numerical experiments, possibly due to the 
symmetry of this composition method, see, e.g., \cite{Ableidinger2017,Stern2020,Chevallier2020,Tubikanec2020}. Thus, the Strang method~\eqref{SP3_FHN} is also considered in the numerical experiments reported in Section \ref{sec:7_FHN}.

\section{Structure preservation}
\label{sec:5_FHN}

The mean-square convergence discussed in the previous section is a limit result for the time discretisation step $\Delta$ going to zero over a finite interval. This result does not carry any information about the quality of the numerical method under the use of strictly positive time steps $\Delta$, as always required when implementing any numerical method. In the following, we discuss the preservation of important structural properties, focusing on hypoellipticity and ergodicity. 

\subsection{Preservation of noise structure and $1$-step hypoellipticity}

To obtain a discrete analogue of the transition probability \eqref{eq:trans_prob} introduced in Section \ref{sec:2_FHN}, we define the $k$-step transition probability of a numerical solution $\widetilde{X}(t_i)$ of SDE \eqref{eq:semi_linear_SDE} as follows
\begin{equation}\label{eq:trans_prob_k}
	\widetilde{P}_{t_k}(\mathcal{A},x):=\mathbb{P}(\widetilde{X}(t_k)\in\mathcal{A}|\widetilde{X}(0)=x),
\end{equation}
where $\mathcal{A}\in\mathcal{B}(\mathbb{R}^d)$ and $x\in \mathbb{R}^d$. Now, assume that SDE \eqref{eq:semi_linear_SDE} is hypoelliptic, i.e., its transition probability \eqref{eq:trans_prob} has a smooth density even though $\Sigma\Sigma^\top$ is not of full rank, see Section \ref{sec:2:1}. We introduce a discrete version of this property in the subsequent definition.
\begin{definition}[$k$-step hypoellipticity]\label{def:k:hypo}
	Let $\widetilde{X}(t_i)$ be a numerical solution of  \eqref{eq:semi_linear_SDE} and $k\in\mathbb{N}$ be the smallest $k$ such that its transition probability \eqref{eq:trans_prob_k} has a smooth density. Then, $\widetilde{X}(t_i)$ is called $k$-step hypoelliptic.
\end{definition}\noindent
This means that the numerical method propagates the noise into the smooth component after $k$ iteration steps. The preservation of this property is not an issue when using the numerical method to simulate paths of the SDE over a large enough time horizon, as standard methods usually satisfy it for some $k$. For example, the Euler-Maruyama  method has been observed to be $2$-step hypoelliptic, see, e.g., Corollary 7.4 in \cite{Mattingly2002}. 

However, the case $k=1$, where we also use the notation
\begin{equation}\label{eq:trans_prob_1}
	\widetilde{P}_{\Delta}(\mathcal{A},x):=\mathbb{P}(\widetilde{X}(t_i)\in\mathcal{A}|\widetilde{X}(t_{i-1})=x),
\end{equation}
is of crucial relevance when using the numerical method within statistical applications. For example, in the field of likelihood-based parameter estimation, explicit numerical methods are used to approximate transition densities \cite{Ditlevsen2019,Melnykova2018,Pokern2009}. In this regard, a particular interest lies in situations where \eqref{eq:trans_prob_1} corresponds to a non-degenerate multivariate normal distribution, i.e., $\widetilde{X}(t_i)$ given $\widetilde{X}(t_{i-1})$ is normally distributed with a covariance matrix that reflects the propagation of the noise to the smooth components. This is not the case for the Euler-Maruyama method \eqref{EM_FHN}, as it yields a degenerate multivariate normal transition distribution with conditional covariance matrix given~by
\vspace{-0.2cm}\begin{eqnarray*}
	\nonumber&&\textrm{Cov}(\widetilde{X}^{\textrm{EM}}(t_i)|\widetilde{X}^{\textrm{EM}}(t_{i-1}))=\Delta\Sigma\Sigma^{\top}.
\end{eqnarray*}
Note that the same degenerate covariance matrix is obtained by tamed/truncated variants of the Euler-Maruyama method (see Section \ref{sec:7:1}).

In contrast, the conditional covariance matrix of the Lie-Trotter splitting \eqref{SP1_FHN} coincides with $C(\Delta)$, as defined in \eqref{eq:Cov}. Thus, if the stochastic linear subequation \eqref{SDE_FHN} of the splitting framework is hypoelliptic, the proposed splitting method yields a non-degenerate multivariate normal transition distribution.
\begin{assumption}\label{assum:A_hypo}
	The matrix $A$ is such that SDE \eqref{SDE_FHN} is hypoelliptic.
\end{assumption}
\begin{theorem}[1-step hypoellipticity]\label{thm:hypo_N_LT}
	Let $\widetilde{X}^{\textrm{LT}}(t_i)$ be the splitting method defined through \eqref{SP1_FHN} and let Assumption~\ref{assum:A_hypo} hold. Then, $\widetilde{X}^{\textrm{LT}}(t_i)$ is $1$-step hypoelliptic according to Definition \ref{def:k:hypo}. Moreover, $\widetilde{X}^{\textrm{LT}}(t_i)$ given $\widetilde{X}^{\textrm{LT}}(t_{i-1})$ admits a non-degenerate  normal distribution with mean vector and covariance matrix given by 
	\begin{equation*}
		\mathbb{E}\left[\widetilde{X}^{\textrm{LT}}(t_i)|\widetilde{X}^{\textrm{LT}}(t_{i-1})\right]=e^{A\Delta}f(\widetilde{X}^{\textrm{LT}}(t_{i-1});\Delta), \quad \textrm{Cov}(\widetilde{X}^{\textrm{LT}}(t_i)|\widetilde{X}^{\textrm{LT}}(t_{i-1}))=C(\Delta),
	\end{equation*}
	respectively, where $C(\Delta)$ is defined in \eqref{eq:Cov}.
\end{theorem} 
\begin{proof}
	The fact that $\widetilde{X}^{\textrm{LT}}(t_i)$ given $\widetilde{X}^{\textrm{LT}}(t_{i-1})$ is normally distributed with the corresponding mean vector and covariance matrix is an immediate consequence of formula \eqref{SP1_FHN}, recalling that the $\xi_i$ are Gaussian random vectors with null mean and covariance matrix $C(\Delta)$. Moreover, the linear SDE~\eqref{SDE_FHN} is hypoelliptic by assumption. Thus, its solution $(X^{[1]}(t))_{t\in[0,T]}$ has conditional covariance matrix $C(t)=\textrm{Cov}(X^{[1]}(t)|X_0^{[1]})$ \eqref{eq:Cov} which is of full rank. Since the covariance matrix of the Lie-Trotter splitting equals $C(\Delta)$, this method is $1$-step hypoelliptic according to Definition~\ref{def:k:hypo}, and thus the normal distribution is non-degenerate.
\end{proof}

\begin{remark}
	Note that the $1$-step hypoellipticity of the Lie-Trotter splitting \eqref{SP1_FHN}  and the fact that this method yields a non-degenerate normal transition distribution according to Theorem~\ref{thm:hypo_N_LT} hold without requiring any conditions on the nonlinearity of SDE \eqref{eq:semi_linear_SDE}. Moreover, even though the transition distribution of the Strang splitting \eqref{SP3_FHN} is not explicitly available in general, this numerical method is expected to be $1$-step hypoelliptic too, since it also benefits from the covariance matrix $C(\Delta)$~\eqref{eq:Cov}.
\end{remark}

\subsection{Preservation of Lyapunov structure and geometric ergodicity}\label{sec:5:2_FHN}

We now assume that SDE \eqref{eq:semi_linear_SDE} is geometrically ergodic. The main task to establish the geometric ergodicity of a numerical solution of \eqref{eq:semi_linear_SDE} is to prove a discrete analogue of the Lyapunov condition~\eqref{eq:Lyapunov} introduced in Section \ref{sec:2:2}.
\begin{definition}[Discrete Lyapunov condition]\label{def:Lyapunov_discrete}
	Let $L$ be a Lyapunov function for SDE~\eqref{eq:semi_linear_SDE}. A numerical solution $\widetilde{X}(t_i)$ of \eqref{eq:semi_linear_SDE} satisfies the discrete Lyapunov condition if there exist $\tilde{\rho}\in (0,1)$ and $\tilde{\eta}\geq0$ such that
	\begin{equation*}
		\mathbb{E}\left[ L(\widetilde{X}(t_i)) | \widetilde{X}(t_{i-1}) \right] \leq \tilde{\rho}L(\widetilde{X}(t_{i-1}))+\tilde{\eta}, \quad \forall \ i\in\mathbb{N}.
	\end{equation*}
\end{definition}\noindent
Analogously to the continuous case, this condition implies geometric ergodicity of the numerical method if SDE \eqref{eq:semi_linear_SDE} is elliptic. If the equation is not elliptic, in addition to the discrete Lyapunov condition, $k$-step hypoellipticity and a discrete irreducibility condition are required. For further details, the reader is referred to \cite{Ableidinger2017,Mattingly2002}. 

Euler-Maruyama type methods do not preserve this property, especially when the drift of SDE \eqref{eq:semi_linear_SDE} is only locally Lipschitz continuous. In particular, the problem does not lie in the preservation of hypoellipticity and irreducibility, but in preserving the Lyapunov structure \cite{Mattingly2002}. Consider, for example, the cubic one-dimensional SDE
\begin{equation}\label{eq:Toy}
	dX(t)=-X^3(t)dt+\sigma dW(t), \quad X(0)=X_0.
\end{equation}
Since $F(x)x=-x^4\leq1-x^4\leq2-2x^2$, this SDE satisfies the dissipativity condition \eqref{eq:dissipative}. Thus, $L(x)=1+x^2$ is a Lyapunov function satisfying \eqref{eq:Lyapunov} and the process $(X(t))_{t\in[0,T]}$ is geometrically ergodic. However, it is shown in Lemma 6.3 of \cite{Mattingly2002} that, if $\mathbb{E}[X_0^2]\geq2/\Delta$, the second moment of the Euler-Maruyama method goes to infinity as the time $t_i$ grows, since
\begin{equation*}
	\mathbb{E}\left[ \left( \widetilde{X}^{\textrm{EM}}(t_i) \right)^2 \right] \geq \mathbb{E}[X_0^2] + t_i.
\end{equation*}
Thus, for any fixed time step $\Delta>0$ (even when it is chosen to be arbitrarily small), one can find a starting value $X_0$ such that the Euler-Maruyama method does not converge to a unique invariant distribution. 
This also means that for any $\Delta>0$ and $X_0$, there is a positive probability of blow-up, as discussed in \cite{Humphries2001}.

In contrast, splitting methods may preserve the Lyapunov structure \cite{Ableidinger2017,Bou-Rabee2010,LeimkuhlerMatthewsStoltz2015}. This is also proved for the proposed splitting method and the Lyapunov function $L(x)=1+\norm{x}^2$, under an additional Assumption on the function $f$ and the matrix~$A$, respectively.
\begin{assumption}\label{assum:f_additional}
	There exists a constant $c_3\geq 0$ such that for any $x\in \mathbb{R}^d$, it holds that
	\begin{equation*}
		\norm{f(x;\Delta)}^2 \leq \norm{x}^2+c_3 \Delta, \quad \forall \Delta \in (0,\Delta_0].
	\end{equation*}
\end{assumption}
\begin{assumption}\label{assum:A_norm}
	The matrix A is such that $\norm{e^{A\Delta}}<1$ for all $\Delta \in (0,\Delta_0]$.
\end{assumption}
\begin{theorem}[Discrete Lyapunov condition]\label{thm:Lyapunov_discrete}
	Let $\widetilde{X}^{\textrm{LT}}(t_i)$ 
	be the splitting method defined through \eqref{SP1_FHN},
	and let Assumptions \ref{assum:f_additional} and \ref{assum:A_norm} hold.  
	Then $\widetilde{X}^{\textrm{LT}}(t_i)$ satisfies the discrete Lyapunov condition of Definition~\ref{def:Lyapunov_discrete} with Lyapunov function $L(x)=~1+\norm{x}^2$.
\end{theorem}
\begin{proof}
	We have that
	\begin{eqnarray*}
		\norm{\widetilde{X}^{\textrm{LT}}(t_i)}^2&=&\norm{e^{A\Delta}f(\widetilde{X}^{\textrm{LT}}(t_{i-1});\Delta)+\xi_{i-1}}^2 \\
		&=&f(\widetilde{X}^{\textrm{LT}}(t_{i-1});\Delta)^\top (e^{A\Delta})^\top (e^{A\Delta}) f(\widetilde{X}^{\textrm{LT}}(t_{i-1});\Delta)\\ && \hspace{0.5cm}+f(\widetilde{X}^{\textrm{LT}}(t_{i-1});\Delta)^\top(e^{A\Delta})^\top \xi_{i-1} + \xi_{i-1}^\top e^{A\Delta} f(\widetilde{X}^{\textrm{LT}}(t_{i-1});\Delta) + \xi_{i-1}^\top \xi_{i-1}.
	\end{eqnarray*}
	Denoting the diagonal entries of the covariance matrix $C(\Delta)$ \eqref{eq:Cov} by $c_{jj}(\Delta)$, taking the expectation, using the fact that $\widetilde{X}^{\textrm{LT}}(t_{i-1})$ and $\xi_{i-1}$ are independent, that $\mathbb{E}[\xi_{i-1}]=~0_d$, that $\mathbb{E}[\xi_{i-1}^\top]=0_d^\top$ and that
	\begin{equation*}
		\bar{C}(\Delta):=\sum\limits_{j=1}^{d}c_{jj}(\Delta)=
		\mathbb{E}[\xi_{i-1}^\top \xi_{i-1}  ],
	\end{equation*}
	we get that
	\begin{eqnarray}
		%\nonumber
		\mathbb{E}\left[ \norm{\widetilde{X}^{\textrm{LT}}(t_i)}^2 \right] 
		&=& \label{eq:help} \mathbb{E}\left[ \norm{ e^{A\Delta} f(\widetilde{X}^{\textrm{LT}}(t_{i-1});\Delta) }^2 \right] + \bar{C}(\Delta).
	\end{eqnarray}	
	Considering
	\begin{equation*}
		L(\widetilde{X}^{\textrm{LT}}(t_i))=1+\norm{\widetilde{X}^{\textrm{LT}}(t_i)}^2,
	\end{equation*}
	and using \eqref{eq:help} and Assumption \ref{assum:f_additional}, we get that
	\begin{eqnarray*}
		\mathbb{E}\left[ L(\widetilde{X}^{\textrm{LT}}(t_i))| \widetilde{X}^{\textrm{LT}}(t_{i-1}) \right]&=&1+\norm{e^{A\Delta}f(\widetilde{X}^{\textrm{LT}}(t_{i-1});\Delta)}^2 + \bar{C}(\Delta) \\
		&\leq& 1+\norm{e^{A\Delta}}^2 \left( \norm{\widetilde{X}^{\textrm{LT}}(t_{i-1})}^2+c_3\Delta \right) + \bar{C}(\Delta)+\norm{e^{A\Delta}}^2 \\
		&=& \norm{e^{A\Delta}}^2 L(\widetilde{X}^{\textrm{LT}}(t_{i-1})) + 1 + \norm{e^{A\Delta}}^2c_3\Delta +\bar{C}(\Delta),
	\end{eqnarray*}
	where we added  $\norm{e^{A\Delta}}^2$ in the inequality. Thus, applying Assumption \ref{assum:A_norm}, the discrete Lyapunov condition of Definition \ref{def:Lyapunov_discrete} is satisfied for
	\begin{equation*}
		\tilde{\rho}=\norm{e^{A\Delta}}^2 <1 \quad \text{and} \quad \tilde{\eta}=1 + \norm{e^{A\Delta}}^2c_3\Delta+\bar{C}(\Delta)>0,
	\end{equation*}
	which proves the result.
\end{proof}
In the following corollary of Theorem \ref{thm:Lyapunov_discrete}, we show that the second moment of the splitting method is asymptotically bounded by a constant which is independent of $T$, $\Delta$ and~$i$. In particular, this bound is reached exponentially fast, independently of the choice of $X_0$, in agreement with the geometric ergodicity of the splitting method. This result also requires Assumption~\ref{assum:f_additional} and an assumption related to the matrix $A$. 
\begin{assumption}\label{assum:A_norm_log}
	The matrix A is such that the logarithmic norm $\mu(A)<0$.
\end{assumption}\noindent
Note that Assumption \ref{assum:A_norm_log} implies Assumption \ref{assum:A_norm}, since $\norm{e^{A\Delta}}\leq e^{\mu(A)\Delta}$ \cite{Strom1975}. However, the converse is not true in general. Assumption \ref{assum:A_norm_log} is, e.g., satisfied for normal matrices, where all eigenvalues have strictly negative real part \cite{Soderlind2006}. Matrices contained in this class are, e.g., diagonal ones with strictly negative diagonal entries.
\begin{corollary}[Asymptotic second moment bound]\label{cor:asym_2nd_moment}
	Let $\widetilde{X}^{\textrm{LT}}(t_i)$ 
	be the splitting method defined through \eqref{SP1_FHN}, 
	and let Assumptions~\ref{assum:f_additional} and \ref{assum:A_norm_log} hold. Then, there exists a constant ${K}^{\textrm{LT}}_\infty>0$, %and $\widetilde{K}^{\textrm{S}}_\infty(\Delta_0)>0$, 
	which is independent of $T$, $\Delta$ and $i$, such that
	\begin{equation*}
		\lim\limits_{t_i\to\infty}\mathbb{E}\left[ \norm{\widetilde{X}^{\textrm{LT}}(t_i)}^2 \right] \leq {K}^{\textrm{LT}}_\infty. 
	\end{equation*}
\end{corollary}
\begin{proof}
	Recalling \eqref{eq:help} from the proof of Theorem \ref{thm:Lyapunov_discrete}, and using Assumption \ref{assum:f_additional} and the logarithmic norm, we further obtain
	\begin{eqnarray*}
		\mathbb{E}\left[ \norm{\widetilde{X}^{\textrm{LT}}(t_i)}^2 \right]&\leq&e^{2\mu(A)\Delta} \left( \mathbb{E}\left[ \norm{\widetilde{X}^{\textrm{LT}}(t_{i-1})}^2 \right]+ c_3\Delta \right) +\bar{C}(\Delta).
	\end{eqnarray*}
	Now, we can perform back iteration, yielding
	\begin{eqnarray*}
		\mathbb{E}\left[ \norm{\widetilde{X}^{\textrm{LT}}(t_i)}^2 \right]&\leq&e^{2\mu(A)t_i} \mathbb{E}\left[ \norm{X_0}^2 \right]+ c_3\Delta \sum\limits_{k=1}^{i} e^{2\mu(A)k\Delta} +\bar{C}(\Delta) \sum\limits_{k=0}^{i-1}e^{2\mu(A)k\Delta}.
	\end{eqnarray*}
	Using that
	%\begin{eqnarray*}
	%	\sum\limits_{k=1}^{i}e^{2\mu(A)k\Delta}&=&\left(1- e^{2\mu(A)t_i} \right)\frac{e^{2\mu(A)\Delta}}{1- e^{2\mu(A)\Delta}},\\
	%	\sum\limits_{k=0}^{i-1}e^{2\mu(A)k\Delta}&=&\left(1- e^{2\mu(A)t_i} \right)\frac{1}{1- e^{2\mu(A)\Delta}},
	%\end{eqnarray*}
	\begin{equation*}
		\sum\limits_{k=1}^{i}e^{2\mu(A)k\Delta}=\left(1- e^{2\mu(A)t_i} \right)\frac{e^{2\mu(A)\Delta}}{1- e^{2\mu(A)\Delta}}, \quad
		\sum\limits_{k=0}^{i-1}e^{2\mu(A)k\Delta}=\left(1- e^{2\mu(A)t_i} \right)\frac{1}{1- e^{2\mu(A)\Delta}},
	\end{equation*}
	we obtain 
	\begin{equation}\label{eq:bound_LT}
		\mathbb{E}\left[ \norm{\widetilde{X}^{\textrm{LT}}(t_i)}^2 \right]\leq e^{2\mu(A)t_i} \mathbb{E}\left[ \norm{X_0}^2 \right]+\left( 1-e^{2\mu(A)t_i} \right)\left( \frac{c_3\Delta e^{2\mu(A)\Delta}}{1-e^{2\mu(A)\Delta}}+\frac{\bar{C}(\Delta)}{1-e^{2\mu(A)\Delta}} \right).
	\end{equation}
	Applying Assumption \ref{assum:A_norm_log}, yields
	\begin{eqnarray*}
		\lim\limits_{t_i\to\infty}\mathbb{E}\left[ \norm{\widetilde{X}^{\textrm{LT}}(t_i)}^2 \right]&\leq&\frac{c_3\Delta e^{2\mu(A)\Delta}}{1-e^{2\mu(A)\Delta}}+\frac{\bar{C}(\Delta)}{1-e^{2\mu(A)\Delta}}. %\\
	\end{eqnarray*}
	Now, we have that 
	\begin{equation}\label{eq:help3}
		\frac{\Delta e^{2\mu(A)\Delta}}{1-e^{2\mu(A)\Delta}} \leq -\frac{1}{2\mu(A)}, \quad \forall \ \Delta>0 \quad \text{and} \quad \frac{\Delta}{1-e^{2\mu(A)\Delta}} \leq \frac{\Delta_0}{1-e^{2\mu(A)\Delta_0}}, \quad \forall \ \Delta \in (0,\Delta_0].
	\end{equation}\noindent
	Moreover, recalling that $e^{A\Delta}=\mathbb{I}_d+\Delta A+O(\Delta^2)$, it follows from \eqref{eq:Cov} that $\bar{C}(\Delta)=O(\Delta)$. This implies the result.
\end{proof}
\begin{remark}
	Theorem \ref{thm:Lyapunov_discrete} and Corollary \ref{cor:asym_2nd_moment} can be proved similarly for the Strang splitting method~\eqref{SP3_FHN}.
\end{remark}

\paragraph*{The one-dimensional case}

Consider the case $d=1$, $\Sigma=\sigma>0$ and $A=-a<0$. In this case, the solution of the linear SDE \eqref{SDE_FHN} corresponds to the Ornstein-Uhlenbeck process 
\begin{equation}\label{eq:OU}
	X^{[1]}(t)=e^{-at}X_0^{[1]}+\sigma \int\limits_{0}^{t}e^{-a(t-s)}dW(s),
\end{equation}
with variance \eqref{eq:Cov} given by
\begin{equation}\label{eq:OU_C}
	C(t)=\frac{\sigma^2}{2a}(1-e^{-2at}).
\end{equation}
Thus, due to the specific form of \eqref{eq:OU_C}, the previously derived bound can be expressed in closed-form for any time $t_i$. In particular, the following (asymptotic) bound for the second moment of the splitting method \eqref{SP1_FHN} is obtained.
\begin{corollary}[Closed-form (asymptotic) second moment bound]\label{cor:bounds_1dim}
	Let $d=1$, $\Sigma=\sigma>0$ and \text{$A=-a<0$}. Further, let  $\widetilde{X}^{\textrm{LT}}(t_i)$ be the  splitting method defined through \eqref{SP1_FHN}, and let
	Assumption~\ref{assum:f_additional} be satisfied. Then, it holds that
	\begin{eqnarray*}
		\mathbb{E}\left[ (\widetilde{X}^{\textrm{LT}}(t_i))^2 \right] &\leq& {K}^{\textrm{LT}}(t_i,X_0):=e^{-2at_i}\mathbb{E}\left[ X_0^2 \right] + (1-e^{-2at_i}) \left(\frac{c_3}{2a}+\frac{\sigma^2}{2a}\right),\\
		\lim\limits_{t_i\to\infty}\mathbb{E}\left[ (\widetilde{X}^{\textrm{LT}}(t_i))^2 \right] &\leq& {K}^{\textrm{LT}}_\infty:=\frac{c_3}{2a}+\frac{\sigma^2}{2a}. 
	\end{eqnarray*}
\end{corollary}
\begin{proof}
	Using \eqref{eq:OU_C} and noting that $\bar{C}(\Delta)=C(\Delta)$ and that $\mu(A)=-a<0$,
	the result is a direct consequence of Corollary~\ref{cor:asym_2nd_moment}.
\end{proof}
Note that, for $t_i=0$, the bound ${K}^{\textrm{LT}}(0,X_0)$ in Corollary \ref{cor:bounds_1dim} coincides with~$\mathbb{E}[X_0^2]$. Moreover, since $a>0$, the distribution of \eqref{eq:OU} converges to a unique limit
\begin{equation}\label{eq:OU_lim}
	X^{[1]}(t) \xrightarrow[t \to \infty]{\mathcal{D}} \mathcal{N}\Bigl(0,\frac{\sigma^2}{2a}\Bigr).
\end{equation}
Intuitively, this fact, combined with Assumption \ref{assum:f_additional}, guarantees the geometric ergodicity of the splitting method obtained via Theorem \ref{thm:Lyapunov_discrete}, and thus the existence of the asymptotic bound for the second moment reported in Corollary \ref{cor:bounds_1dim}.

\paragraph*{Cubic model problem}

For an illustration of the derived bound, consider again SDE \eqref{eq:Toy}. 
We propose to rewrite this equation as
\begin{equation*}
	dX(t)=\left( -X(t) +X(t) -X^3(t)\right)dt+\sigma dW(t),
\end{equation*}
and to choose 
\begin{equation}\label{eq:A_N_Toy}
	A=-1<0, \qquad N(X(t))=X(t)-X^3(t).
\end{equation}
The exact solution of the resulting linear SDE \eqref{SDE_FHN} is then given by \eqref{eq:OU} for $a=1$, and that 
of ODE~\eqref{ODE_FHN} is given by
\begin{equation}\label{eq:f_toy}
	X^{[2]}(t)=f(X_0^{[2]};t)=\frac{X_0^{[2]}}{\sqrt{e^{-2t}+(X_0^{[2]})^2(1-e^{-2t})}}.
\end{equation}
This choice guarantees that all required assumptions are satisfied.
\begin{proposition}\label{prop:A1_A2_N_toy}
	Let $A$, $N$ and $f$ be as in \eqref{eq:A_N_Toy} and \eqref{eq:f_toy}, respectively. Then, Assumptions \ref{assum:1}--\ref{assum:A_norm_log} are satisfied.
\end{proposition}
\begin{proof}
	The proof is given in Appendix \ref{appA_FHN}.
\end{proof}
Therefore, the proposed splitting method \eqref{SP1_FHN} applied to SDE \eqref{eq:Toy} is not only mean-square convergent, but also geometrically ergodic. In particular, while even for arbitrarily small $\Delta$ one can find $X_0$ such that the second moment of the Euler-Maruyama method explodes (see the beginning of Section \ref{sec:5:2_FHN}), the second moment of the splitting method is bounded by ${K}^{\textrm{LT}}(t_i,X_0)$, which converges to the constant ${K}^{\textrm{LT}}_\infty={1}/{4}+{\sigma^2}/{2}$ exponentially fast and for any choice of the initial value $X_0$, see Corollary \ref{cor:bounds_1dim}.

In Figure \ref{fig:bound_toy}, we illustrate the derived second moment bound ${K}^{\textrm{LT}}(t_i,X_0)$ (grey solid line) of Corollary \ref{cor:bounds_1dim} for SDE \eqref{eq:Toy} as a function of the time $t_i$, in comparison with $\mathbb{E}[X^2(t_i)]$ (red dashed line). The latter is estimated based on $10^4$ paths generated under the Lie-Trotter splitting \eqref{SP1_FHN}. The asymptotic bound ${K}^{\textrm{LT}}_\infty$ of Corollary \ref{cor:bounds_1dim} is indicated by the blue dotted line. 

\begin{remark}
	For SDE \eqref{eq:Toy}, an immediate choice of the subequations of the splitting framework would also be $N(X(t))=-X^3(t)$ and $A=0$. For this choice, Assumption \ref{assum:1} related to the locally Lipschitz function $N$ is satisfied, and thus the resulting splitting method $\eqref{SP1_FHN}$ is mean-square convergent. Also, Assumptions \ref{assum:A_hypo} and \ref{assum:f_additional} are satisfied. However, since $e^{A\Delta}=1$ and $\mu(A)=0$, Assumptions \ref{assum:A_norm} and \ref{assum:A_norm_log} do not hold, asymptotic bounds cannot be derived, and the preservation of ergodicity remains an open question. In particular, in contrast to the proposed approach (see formulas \eqref{eq:OU} and \eqref{eq:OU_lim}), the distribution of the solution $X^{[1]}(t)=X_0^{[1]}+\sigma W(t)$ of the resulting linear SDE~\eqref{SDE_FHN} does not converge to a unique limit as $t$ tends to infinity.
\end{remark}

\begin{figure}
	\begin{centering}
		\includegraphics[width=0.6\textwidth]{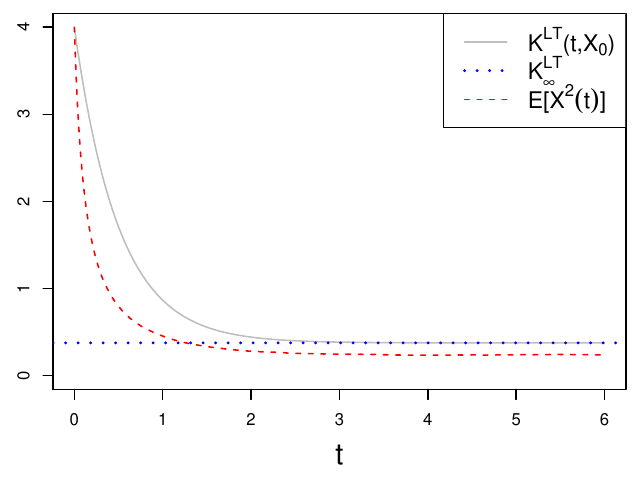}	
		\caption{Bound ${K}^{\textrm{LT}}(t_i,X_0)$ (as a function of time) and asymptotic bound ${K}^{\textrm{LT}}_\infty$ of Corollary \ref{cor:bounds_1dim} for SDE \eqref{eq:Toy} ($a=1 $ and $c_3=~1/2$) with $\sigma=1/2$ and $X_0=2$, and estimate of $\mathbb{E}[X^2(t_i)]$ obtained from $10^4$ paths generated under the LT splitting.}
		\label{fig:bound_toy}
	\end{centering}
\end{figure}

\section{Stochastic FitzHugh-Nagumo model}
\label{sec:6_FHN}

In this section, the proposed splitting strategy is illustrated on the stochastic FHN model, a widely used neuronal model. It is given by the $2$-dimensional SDE \eqref{FHN} with solution $X(t):=(V(t),U(t))^{\top}$ for $t \in [0,T]$.
This equation has been used to model the firing activity of single neurons \cite{FitzHugh1961,Nagumo1962}. If the membrane voltage of the neuron is sufficiently high, it releases an action potential, also called spike. The first component $(V(t))_{t \in [0,T]}$ describes the membrane voltage of the neuron at time $t$, while the second component $(U(t))_{t \in [0,T]}$ corresponds to a recovery variable modelling the channel kinetics. The parameter $\epsilon>0$ corresponds to the time scale separation of the two components and $\beta\geq 0$ and $\gamma>0$ are position and duration parameters of an excitation, respectively.

\subsection{Properties of the FHN model}
\label{sec:6:1}

If both noise intensities $\sigma_1$ and $\sigma_2$ are strictly positive, the model is elliptic. If $\sigma_1=0$, the diffusion term becomes $\Sigma dW(t)=(0,\sigma_2)^\top dW_2(t)$, corresponding to the notation in \eqref{eq:hypo_SDE}. In this case, due to the $U$-component entering the first entry of the drift $F(X(t))$, the model is hypoelliptic. This is confirmed by the fact that 
\begin{equation}\label{eq:cond_hypo_FHN}
	\partial_u F_1(x) \sigma_2=-\frac{\sigma_2}{\epsilon}  \neq 0,
\end{equation}
guaranteeing Condition \eqref{eq:cond_hypo}. We refer to \cite{Berglund2012,Bonaccorsi2008,Kelly2018,Muratov2008} and to \cite{Ditlevsen2019,Leon2018} for the consideration of the elliptic and hypoelliptic FHN model, respectively, and to \cite{Quentin2020} for an investigation of both cases.

Moreover, it has been proved that the FHN model is ergodic, see, e.g., \cite{Bonaccorsi2008,Leon2018}. Here, we study this property under a restricted parameter space, for which SDE \eqref{FHN} satisfies the dissipativity condition \eqref{eq:dissipative} such that the function $L(x)=1+\norm{x}^2$ is a Lyapunov function meeting Condition~\eqref{eq:Lyapunov}.
\begin{proposition}[Dissipativity of the FHN model]\label{prop:dissipative_FHN}
	Let 
	\begin{equation*}
		\left|\gamma-\frac{1}{\epsilon}\right|< 2\min\left\{ \frac{1}{\epsilon},1-\tau \right\},
	\end{equation*}
	for some arbitrarily small $\tau \in (0,1)$.
	Then, the drift $F$ of the FHN model \eqref{FHN} satisfies the dissipativity condition \eqref{eq:dissipative}.
\end{proposition}
\begin{proof}
	We have that
	\begin{eqnarray*}
		(F(x),x)=\Bigl( \begin{pmatrix}
			\frac{1}{\epsilon}(v-v^3-u) \\
			\gamma v-u+\beta
		\end{pmatrix}, \begin{pmatrix}
			v \\ u
		\end{pmatrix} \Bigr)
		= \frac{1}{\epsilon}(v^2-v^4)+vu(\gamma-\frac{1}{\epsilon})-u^2+\beta u.
	\end{eqnarray*}
	Defining $c:=|\gamma-1/\epsilon|$, using  $2vu\leq v^2+u^2$ and $v^2-v^4\leq1-v^2$, applying Young's inequality $\beta u \leq u^2\bar{\tau}/2 +\beta^2/2\bar{\tau}$, for some arbitrarily small $\bar{\tau}>0$, and setting $\tau=\bar{\tau}/2$, we obtain
	\begin{eqnarray*}
		(F(x),x) &\leq& \frac{1}{\epsilon}(1-v^2)+\frac{c}{2}(v^2+u^2)-u^2+{\tau}u^2+\frac{\beta^2}{2\bar{\tau}} \\
		&=& -v^2(\frac{1}{\epsilon}-\frac{c}{2})-u^2(1-{\tau}-\frac{c}{2})+\frac{1}{\epsilon}+\frac{\beta^2}{2\bar{\tau}}.
	\end{eqnarray*}
	Since
	\begin{equation*}
		\frac{1}{\epsilon}-\frac{c}{2}>0 \quad \text{and} \quad 1-{\tau}-\frac{c}{2}>0,
	\end{equation*}
	by assumption, it follows that
	\begin{equation*}
		(F(x),x)\leq \alpha-\delta\norm{x}^2,
	\end{equation*}
	where $\alpha=1/\epsilon+{\beta^2}/{2\bar{\tau}}>0$ and $\delta=\min\{ 1/\epsilon-c/2,1-\tau-c/2 \}>0$.
\end{proof}\noindent
Note that the condition on the model parameters in Proposition \ref{prop:dissipative_FHN} is satisfied for parameter settings which may be relevant in applications, see Section \ref{sec:7_FHN}. For example, it is met  when $\gamma=1/\epsilon$.

\subsection{Splitting method for the FHN model}

The FHN model \eqref{FHN} is a semi-linear SDE of type \eqref{eq:semi_linear_SDE}. The choice of the matrix $A$ and the function $N(X(t))$ is not unique. While the locally Lipschitz term $-V^3(t)/\epsilon$ and the constant $\beta$ have to enter into $N(X(t))$, the goal is to allocate the remaining terms such that as many of the introduced assumptions as possible are satisfied. For the splitting method to satisfy Assumption~\ref{assum:A_hypo}, and thus to be $1$-step hypoelliptic, the term $-U(t)/\epsilon$ of the first component of the drift must enter into $AX(t)$. Moreover, shifting the term $\gamma V(t)$ to $AX(t)$ leads to a decoupling of the resulting ODE \eqref{ODE_FHN} such that its global solution can be derived exactly and proved to satisfy Assumption~\ref{assum:f_additional}. Thus, there are four strategies left, depending on whether the remaining terms $V(t)/\epsilon$ and $-U(t)$ enter into $AX(t)$ or $N(X(t))$. The only case where the matrix $A$ meets Assumption \ref{assum:A_norm} (under a restricted parameter space) is when $-U(t)$ appears in $AX(t)$ and $V(t)/\epsilon$ in $N(X(t))$. Similar to the proposed splitting of SDE \eqref{eq:Toy}, the resulting linear SDE \eqref{SDE_FHN} is then geometrically ergodic. In particular, it corresponds to a version of the well-studied damped stochastic harmonic oscillator whose matrix exponential $e^{At}$ and covariance matrix $C(t)$ have manageable expressions. Therefore, we propose to choose the matrix $A$ and the function $N$ as follows
\begin{equation}\label{eq:A_N_FHN}
	A=\begin{pmatrix}
		0 \ & \ -\frac{1}{\epsilon} \\
		\gamma \ & \ -1
	\end{pmatrix}, \qquad N(X(t))=\begin{pmatrix}
		\frac{1}{\epsilon}\bigl( V(t)-V^3(t) \bigr) \\
		\beta
	\end{pmatrix}.
\end{equation}

The resulting linear damped stochastic harmonic oscillator \eqref{SDE_FHN} with $A$ as in \eqref{eq:A_N_FHN} is weakly-, critically- or over-damped, depending on whether
\begin{equation}\label{kappa}
	\kappa:=\frac{4\gamma}{\epsilon}-1
\end{equation}
is positive, zero or negative, respectively. This terminology, along with the choice of $\kappa$, is linked to the roots of the characteristic function of the underlying deterministic equation, see, e.g., Chapter~$5$ in \cite{Weiglhofer1999}. In particular, the sign of $\kappa$ determines the shape of the exponential of the matrix $A$.
If~$\kappa=0$, 
\begin{equation*}\label{M1}
	e^{At}= e^{-{\frac{t}{2}}} \begin{pmatrix}
		1+\frac{t}{2} & -\frac{t}{\epsilon} \\
		\frac{\epsilon t}{4} & 1-\frac{t}{2}
	\end{pmatrix}.
\end{equation*}
If $\kappa>0$, 
\begin{equation*}\label{M2}
	e^{At}=e^{-\frac{t}{2}}\begin{pmatrix}
		\cos(\frac{1}{2}\sqrt{\kappa}t)+\frac{1}{\sqrt{\kappa}}\sin(\frac{1}{2}\sqrt{\kappa}t) & -\frac{2}{\epsilon\sqrt{\kappa}}\sin(\frac{1}{2}\sqrt{\kappa}t) \\
		\frac{2\gamma}{\sqrt{\kappa}}\sin(\frac{1}{2}\sqrt{\kappa}t) & \cos(\frac{1}{2}\sqrt{\kappa}t)-\frac{1}{\sqrt{\kappa}}\sin(\frac{1}{2}\sqrt{\kappa}t)
	\end{pmatrix}.
\end{equation*}
If $\kappa<0$, the sine and cosine terms of the above expressions can be rearranged using the relations 
\begin{equation}\label{relations}
	\cos\left(\frac{1}{2}\sqrt{\kappa}t\right)=\cosh\left(\frac{1}{2}\sqrt{-\kappa}t\right) \quad \text{and} \quad \frac{1}{\sqrt{\kappa}}\sin\left(\frac{1}{2}\sqrt{\kappa}t\right)=\frac{1}{\sqrt{-\kappa}}\sinh\left(\frac{1}{2}\sqrt{-\kappa}t\right).
\end{equation}\noindent
Moreover, the covariance matrix $C(t)$ \eqref{eq:Cov} also depends on the sign of $\kappa$ and is given as follows.
If $\kappa=0$,
\begin{eqnarray*}
	c_{11}(t)&=& \frac{e^{-t}}{4\epsilon^2}\left( 4 \sigma_2^2 \left( -2+2e^t-t(2+t) \right) + \epsilon^2\sigma_1^2 \left( -10+10e^t-t(6+t) \right) \right), \\
	c_{12}(t)&=&c_{21}(t)=\frac{e^{-t}}{8\epsilon}\left( -4\sigma_2^2t^2+\epsilon^2\sigma_1^2 \left( 4e^t-(2+t)^2 \right) \right), \\
	c_{22}(t)&=&\frac{e^{-t}}{16}\left( 4\sigma_2^2 \left( -2+2e^t-(t-2)t \right) + \epsilon^2 \sigma_1^2 \left( -2+2e^t-t(2+t) \right) \right).
\end{eqnarray*}
If $\kappa>0$,
\begin{eqnarray*}
	c_{11}(t)&=&\frac{\epsilon e^{-t}}{2\gamma\kappa} \biggl( -\frac{4\gamma}{\epsilon^2}( \sigma_1^2\gamma+\sigma_2^2\frac{1}{\epsilon})  + \kappa e^t (\sigma_1^2(1+\frac{\gamma}{\epsilon}) +\sigma_2^2\frac{1}{\epsilon^2}) \\ \nonumber && \hspace{0.5cm} + \Bigl( \sigma_1^2(1-\frac{3\gamma}{\epsilon}) +\sigma_2^2\frac{1}{\epsilon^2} \Bigr)\cos(\sqrt{\kappa}t) - \sqrt{\kappa} (\sigma_1^2(1-\frac{\gamma }{\epsilon})+\sigma_2^2\frac{1}{\epsilon^2}) \sin(\sqrt{\kappa}t) \biggr), \\
	c_{12}(t)&=&c_{21}(t)=\frac{\epsilon e^{-t}}{2 \kappa } \biggl( \sigma_1^2\kappa e^t - \frac{2}{\epsilon}(\sigma_1^2 \gamma  +\sigma_2^2\frac{1}{\epsilon}) \\ \nonumber && \hspace{0.5cm} + \Bigl( \sigma_1^2(1-\frac{2\gamma}{\epsilon}) + 2\sigma_2^2 \frac{1}{\epsilon^2} \Bigr) \cos(\sqrt{\kappa}t)  -\sigma_1^2\sqrt{\kappa} \sin(\sqrt{\kappa}t)  \biggr), \\
	c_{22}(t)&=&\frac{\epsilon e^{-t}}{2\kappa} \biggl( (\sigma_2^2\frac{1}{\epsilon}+\sigma_1^2\gamma) \Bigl(  \cos(\sqrt{\kappa}t)-\frac{4\gamma}{\epsilon}+\kappa e^t \Bigr) + (\sigma_2^2\frac{1}{\epsilon}-\sigma_1^2\gamma) \sqrt{\kappa} \sin(\sqrt{\kappa}t) \biggr).
\end{eqnarray*}
If $\kappa<0$, the relations \eqref{relations} can again be used to rewrite the above expressions accordingly. Note that parameter configurations typically considered in the literature fulfill $\kappa>0$, see, e.g., \cite{Ditlevsen2019,Leon2018,Quentin2020}. This is in agreement with the fact that, under $\kappa>0$, SDE \eqref{SDE_FHN} models a weakly damped system which describes oscillatory dynamics.

The exact solution of the resulting ODE \eqref{ODE_FHN} with $N(X(t))$ as in \eqref{eq:A_N_FHN} reads as
\begin{equation}
	X^{[2]}(t)=f(X_0^{[2]};t)=
	\begin{pmatrix}
		\frac{V_0^{[2]}}{\sqrt{e^{-\frac{2t}{\epsilon}}+(V_0^{[2]})^2\left(1-e^{-\frac{2t}{\epsilon}}\right)}} \\
		\beta t + U_0^{[2]}
	\end{pmatrix}.
	\label{Exact_ODE_FHN2}
\end{equation}
The corresponding Lie-Trotter splitting method for the FHN model \eqref{FHN} is then given by \eqref{SP1_FHN}, where the matrix exponential $e^{A\Delta}$, the covariance matrix $C(\Delta)$ and the function $f$ are as reported above.

\subsection{Properties of the splitting method for the FHN model}

In the following proposition, we verify Assumptions \ref{assum:1}--\ref{assum:A_norm}. 
\begin{proposition}\label{prop:A1_A2_N_FHN}
	Let $A$, $N$ and $f$ be as in \eqref{eq:A_N_FHN} and \eqref{Exact_ODE_FHN2}, respectively. Then the following statements hold.
	\begin{itemize}
		\item[(i)] $N$ satisfies Assumption \ref{assum:1}.
		\item[(ii)] $A$ satisfies Assumption \ref{assum:A_hypo}.
		\item[(iii)] If $\beta=0$, then $f$ satisfies Assumption \ref{assum:f_additional}.
		\item[(iv)] If $\gamma=1/\epsilon$, then $A$ satisfies Assumption \ref{assum:A_norm}.
	\end{itemize}	
\end{proposition}
\begin{proof}
	The proof is given in Appendix \ref{appC_FHN}.
\end{proof}\noindent
Therefore, the proposed splitting method \eqref{SP1_FHN} applied to the FHN model \eqref{FHN} is mean-square convergent of order $1$, according to Theorem \ref{thm:convergence_FHN}. 

Applying Theorem \ref{thm:hypo_N_LT}, the method is also $1$-step hypoelliptic and yields a non-degenerate Gaussian distribution with covariance matrix $C(\Delta)$ reported above. This matrix is thus of full rank, even if $\sigma_1=0$ and independently of the value of $\kappa$. 

Moreover, for $\beta=0$ and $\gamma=1/\epsilon$, $L(x)=1+\norm{x}^2$ is a Lyapunov function for the FHN model \eqref{FHN} according to Propositon \ref{prop:dissipative_FHN} and the method satisfies a discrete Lyapunov condition via Theorem~\ref{thm:Lyapunov_discrete}. Combined with the $1$-step hypoellipticity and a discrete irreducibility condition, which can be proved in the same way as done, e.g, in \cite{Ableidinger2017,Chevallier2020,Mattingly2002}, the splitting method is geometrically ergodic. 
Intuitively, the Lyapunov structure of the FHN model is kept by the numerical solution, since the linear SDE \eqref{SDE_FHN} determined by the matrix $A$ in \eqref{eq:A_N_FHN} is geometrically ergodic, implying that the process $(X^{[1]}(t))_{t \in [0,T]}$ converges to a unique invariant distribution given by
\begin{equation*}
	X^{[1]}(t) \xrightarrow[t \to \infty]{\mathcal{D}} \mathcal{N}\Bigl(
	\begin{pmatrix}
		0 \\
		0 
	\end{pmatrix},
	\begin{pmatrix}
		\frac{5 }{2}\sigma_1^2+\frac{2}{\epsilon^2} \sigma_2^2 & \frac{\epsilon}{2} \sigma_1^2 \\
		\frac{\epsilon}{2} \sigma_1^2 & \frac{\epsilon^2}{8} \sigma_1^2+\frac{1}{2}\sigma_2^2
	\end{pmatrix}
	\Bigr),
\end{equation*}
for $\kappa=0$, and
\begin{equation*}
	X^{[1]}(t) \xrightarrow[t \to \infty]{\mathcal{D}} \mathcal{N}\Bigl(
	\begin{pmatrix}
		0 \\
		0 
	\end{pmatrix},
	\begin{pmatrix}
		\frac{\epsilon}{2\gamma}(\sigma_1^2+\frac{\gamma}{\epsilon}\sigma_1^2+\frac{1}{\epsilon^2}\sigma_2^2) & \frac{\epsilon}{2} \sigma_1^2 \\
		\frac{\epsilon}{2} \sigma_1^2 & \frac{1}{2}(\epsilon\gamma\sigma_1^2+\sigma_2^2)
	\end{pmatrix}
	\Bigr),
\end{equation*}
for $\kappa \neq 0$.
Since this fact holds without any restrictions of the parameters, it is expected that the splitting method preserves this property for any values of $\gamma>0$ and $\epsilon>0$. This is confirmed by our numerical experiments (see Section \ref{sec:7_FHN}).

Note also that, under $\gamma=1/\epsilon$, the logarithmic norm $\mu(A)=0$. Thus, Assumption \ref{assum:A_norm_log} is not fulfilled and the asymptotic bound of Corollary \ref{cor:asym_2nd_moment} cannot be derived.

\begin{remark}
	Another plausible choice of the subequations is
	\begin{equation*}
		A=\begin{pmatrix}
			0 \ & \ -\frac{1}{\epsilon} \\
			\gamma \ & \ \ \ 0
		\end{pmatrix}, \qquad N(X(t))=\begin{pmatrix}
			\frac{1}{\epsilon}\bigl( V(t)-V^3(t) \bigr) \\
			-U(t)+\beta
		\end{pmatrix}.
	\end{equation*}
	For this choice, Assumption \ref{assum:1} related to the locally Lipschitz function $N$ is satisfied, and the splitting method is mean-square convergent. In addition,  since the term $-U(t)/\epsilon$ still enters into $AX(t)$, Assumption \ref{assum:A_hypo} holds, and the method is thus also $1$-step hypoelliptic. Moreover, Assumption~\ref{assum:f_additional} would also hold under $\beta=0$. However, in this case, the linear SDE \eqref{SDE_FHN} corresponds to a version of the simple (undamped) harmonic oscillator which is not ergodic. In particular, its  matrix exponential is given by
	\begin{equation*}
		e^{At}=\begin{pmatrix}
			\cos(\frac{\sqrt{\gamma}t}{\sqrt{\epsilon}}) & -\frac{1}{\sqrt{\epsilon\gamma}}\sin(\frac{\sqrt{\gamma}t}{\sqrt{\epsilon}}) \\
			\sqrt{\epsilon\gamma}\sin(\frac{\sqrt{\gamma}t}{\sqrt{\epsilon}}) & \cos(\frac{\sqrt{\gamma}t}{\sqrt{\epsilon}})
		\end{pmatrix},
	\end{equation*}
	with $\norm{e^{A\Delta}}\geq 1$ and $\norm{e^{A\Delta}}=1$ for $\gamma=1/\epsilon$ in particular. 
\end{remark}

\section{Numerical experiments for the FHN model}
\label{sec:7_FHN}

We now illustrate the performance of the Lie-Trotter \eqref{SP1_FHN} and Strang \eqref{SP3_FHN} splitting methods in comparison with Euler-Maruyama type methods through a variety of numerical experiments carried out on the FHN model~\eqref{FHN}. First, the proved mean-square convergence order $1$ is illustrated numerically. Second, the ability of the different numerical methods to preserve the qualitative dynamics of neuronal spiking is analysed, in particular their ability to reproduce the correct amplitudes and frequencies of the underlying oscillations when the time step $\Delta$ is increased. Third, the robustness of the numerical methods to changes in the initial condition $X_0$, and how the choice of $X_0$ may influence the preservation of the phases of the underlying oscillations are analysed. All simulations are carried out in the computing environment R \cite{R}. Before we present the simulation results, different Euler-Maruyama type comparison methods, proposed for superlinearly growing coefficients, are recalled.

\subsection{Revision of Euler-Maruyama type methods}
\label{sec:7:1}

In \cite{Hutzenthaler2011}, it has been shown that the Euler-Maruyama method \eqref{EM_FHN} is not mean-square convergent if at least one of the coefficients of the SDE grows superlinearly, 
as this results in unbounded moments of the iterates. Since then, 
several explicit variants of this method have been proposed, which aim to control this unbounded growth.

The first variant, designed for polynomially growing and one-sided Lipschitz drift and globally Lipschitz diffusion coefficients, has been introduced in \cite{Hutzenthaler2012_2}. 
It is based on a taming perturbation which avoids large values caused by the superlinearly growing drift. The method is defined through the iteration
\begin{equation}\label{eq:EM_Tamed}
	\widetilde{X}^{\textrm{TEM}}(t_i)=\widetilde{X}^{\textrm{TEM}}(t_{i-1})+\frac{F(\widetilde{X}^{\textrm{TEM}}(t_{i-1}))\Delta }{1+ \norm{F(\widetilde{X}^{\textrm{TEM}}(t_{i-1}))}\Delta} + \Sigma \sqrt{\Delta} \psi_{i-1},
\end{equation}
and proved to be mean-square convergent of order $1/2$ (order $1$) for SDEs with multiplicative noise (additive noise), i.e., it yields the same convergence rate as achieved by the Euler-Maruyama method in the globally Lipschitz case \cite{Kloeden1992}.

Another variant, aiming to tame both the drift and the diffusion term, has been suggested in \cite{Tretyakov2012}. The method is defined via
\begin{equation}\label{DTEM}
	\widetilde{X}^{\textrm{DTEM}}(t_i)=\widetilde{X}^{\textrm{DTEM}}(t_{i-1})+\frac{F(\widetilde{X}^{\textrm{DTEM}}(t_{i-1}))\Delta + \Sigma \sqrt{\Delta} \psi_{i-1} }{1+ \norm{F(\widetilde{X}^{\textrm{DTEM}}(t_{i-1}))}\Delta + \norm{\Sigma \sqrt{\Delta} \psi_{i-1}}},
\end{equation}
and is designed for the broader class of equations where also the diffusion coefficient is allowed to grow polynomially at infinity and satisfies a one-sided Lipschitz condition. It has been shown to converge with mean-square order $1/2$ (also in the case of additive noise). For similar variants of the Euler-Maruyama method, see, e.g, \cite{Sabanis2013,Zhang2017}.

The strong convergence (without order) of a different class of variants, based on space truncation techniques, has been discussed in \cite{Hutzenthaler2012}.
In particular, we recall the two methods
\begin{equation}\label{TrEM}
	\widetilde{X}^{\textrm{TrEM}}(t_i)=\widetilde{X}^{\textrm{TrEM}}(t_{i-1})+\frac{F(\widetilde{X}^{\textrm{TrEM}}(t_{i-1}))\Delta }{\textrm{max}\left\{ 1,\norm{F(\widetilde{X}^{\textrm{TrEM}}(t_{i-1}))}\Delta \right\}} + \Sigma \sqrt{\Delta} \psi_{i-1},
\end{equation}
\begin{equation}\label{DTrEM}
	\widetilde{X}^{\textrm{DTrEM}}(t_i)=\widetilde{X}^{\textrm{DTrEM}}(t_{i-1})+\frac{F(\widetilde{X}^{\textrm{DTrEM}}(t_{i-1}))\Delta + \Sigma \sqrt{\Delta} \psi_{i-1}}{\textrm{max}\left\{ 1,\Delta \norm{F(\widetilde{X}^{\textrm{DTrEM}}(t_{i-1}))\Delta+ \Sigma \sqrt{\Delta} \psi_{i-1}} \right\}},
\end{equation}
constructed to truncate the drift and the drift and diffusion term, respectively.

Another type of truncated Euler-Maruyama method, with mean-square convergent rate arbitrarily close to $1$, has been proposed in \cite{Mao2015,Mao2016}. Here, we recall the partially truncated variant discussed in \cite{Mao2017}. This method assumes that the drift can be decomposed as 
\begin{equation*}
	F(X(t))=F_1(X(t))+F_2(X(t)),
\end{equation*}
where $F_1$ is globally Lipschitz continuous and $F_2$ satisfies Assumption \ref{assum:1}. It is given by
\begin{equation}\label{PTrEM}	
	\widetilde{X}^{\textrm{PTrEM}}(t_i)=\widetilde{X}^{\textrm{PTrEM}}(t_{i-1})+ \left( F_1(\widetilde{X}^{\textrm{PTrEM}}(t_{i-1})) + F_2^\Delta (\widetilde{X}^{\textrm{PTrEM}}(t_{i-1})) \right) \Delta + \Sigma \sqrt{\Delta} \psi_{i-1},
\end{equation}
where the function $F_2^\Delta$ is a truncated version of $F_2$. In particular, it is given by
\begin{equation*}
	F_2^\Delta(x)=F_2\left( \min\{ \norm{x},\mu^{-1} \bigl(h(\Delta)\bigr) \} \frac{x}{\norm{x}} \right),
\end{equation*}
where $\mu:\mathbb{R}_+\to\mathbb{R}_+$ such that $\mu(r)\to \infty$ as $r\to \infty$ and
\begin{equation*}
	\sup_{\norm{x}\leq r} \Bigl( \norm{F_2(x)} \Bigr) \leq \mu(r), \quad \forall \ r\geq 1,
\end{equation*}
and, for $\Delta^* \in (0,1]$, $h:(0,\Delta^*] \to (0,\infty)$ such that
\begin{equation*}
	h(\Delta^*) \geq \mu(1), \quad \lim\limits_{\Delta \to 0}h(\Delta) =\infty \quad \text{and} \quad \Delta^{1/4}h(\Delta)\leq 1, \ \forall \Delta \in (0,1).
\end{equation*}
Thus, the method is not uniquely defined and depends on the choice of $\mu(\cdot)$ and $h(\cdot)$. Following \cite{Mao2017}, for the cubic model problem \eqref{eq:Toy}, we consider $F_1\equiv 0$, $\mu(r)=r^3$ and $h(\Delta)=\Delta^{-1/5}$. For this choice, the above conditions on $h$ are satisfied for $\Delta^*=1$, since $h(\Delta^*)=\mu(1)=1$. Moreover, it holds that $\mu^{-1}(h(\Delta))=\Delta^{-1/15}$. Numerical experiments for the cubic model problem are reported in Appendix \ref{app_cubic}. For the FHN model \eqref{FHN}, we consider
\begin{equation*}
	F_1(X(t))=\begin{pmatrix}
		\frac{1}{\epsilon}\Bigl(V(t)-U(t)\Bigr) \\
		\gamma V(t)-U(t)+\beta
	\end{pmatrix}, \qquad 
	F_2(X(t))=\begin{pmatrix}
		-\frac{1}{\epsilon}V^3(t) \\
		0
	\end{pmatrix},
\end{equation*}
$\mu(r)=r^3/\epsilon$ and $h(\Delta)=\Delta^{-1/5}$. For this choice, the conditions on $h$ are satisfied for $\Delta^*=1/\epsilon^{-5}$, since then $h(\Delta^*)=1/\epsilon=\mu(1)$. Therefore, when $\epsilon$ is small, this method requires very small time steps $\Delta$, and is thus highly inefficient (see the subsequent sections). 

In the following, we denote by tamed (TEM), diffusion tamed (DTEM), truncated (TrEM), diffusion truncated (DTrEM) and partially truncated (PTrEM) Euler-Maruyama method, the schemes \eqref{eq:EM_Tamed}, \eqref{DTEM}, \eqref{TrEM}, \eqref{DTrEM} and \eqref{PTrEM}, respectively.

\subsection{Convergence order}
\label{sec:7:3}

The mean-square convergence order can be illustrated by approximating the left-hand side of the inequality in Theorem \ref{thm:Zang} (for a fixed time $T$ and $p=1$) with the root mean-squared error (RMSE) defined by
\begin{equation}\label{eq:RMSE}
	\text{RMSE}(\Delta):=\left( \frac{1}{M} \sum_{l=1}^{M} \norm{X^l(T)-\widetilde{X}^l_\Delta(T)}^2  \right)^{1/2},
\end{equation}
where $X^l(T)$ and $\widetilde{X}^l_\Delta(T)$ denote the $l$-th simulated path at a fixed time $T$ of the true process and the approximated process, respectively, for $l=1,\ldots,M$.

\begin{figure}
	\begin{centering}
		\includegraphics[width=0.65\textwidth]{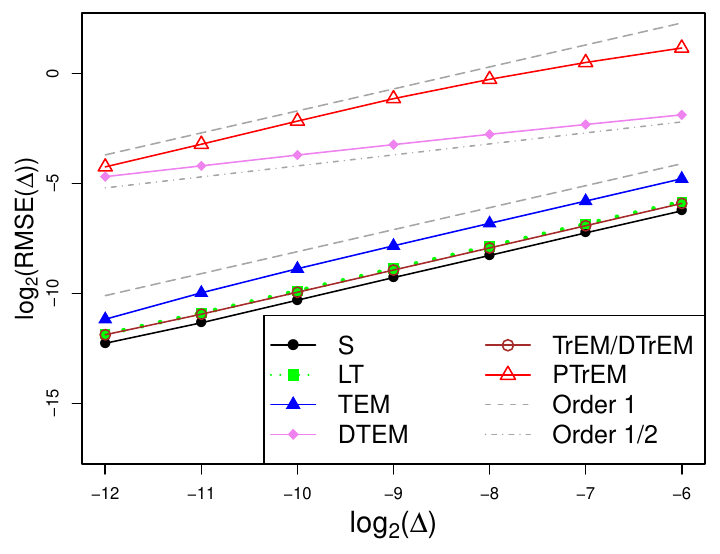}	
		\caption{Illustration of the mean-square convergence order on the FHN model \eqref{FHN} via the RMSE \eqref{eq:RMSE}. All model parameters are set to $1$, $X_0=(0,0)^\top$ and $T=5$.}
		\label{fig:convergence}
	\end{centering}
\end{figure} 

In Figure \ref{fig:convergence}, we report the RMSEs of the different numerical methods in log2 scale as a function of the time step $\Delta$ used to simulate $\widetilde{X}^l_\Delta(T)$. Since the true process is not known, the reference values $X^l(T)$ are simulated with the TEM method \eqref{eq:EM_Tamed} using the small time step $\Delta=2^{-14}$. We verified that using a different scheme for the simulation of the reference paths does not affect the results of the experiments. The approximated trajectories $\widetilde{X}^l_\Delta(T)$ are generated with the LT~\eqref{SP1_FHN}, S \eqref{SP3_FHN}, TEM \eqref{eq:EM_Tamed}, DTEM \eqref{DTEM}, TrEM \eqref{TrEM}, DTrEM~\eqref{DTrEM} and PTrEM~\eqref{PTrEM} method, respectively, under different choices of the time step, namely $\Delta=2^{-k}$, $k=6,\ldots,12$. We consider $T=5$, $M=10^4$, $X_0=(0,0)^{\top}$ and set all model parameters to $1$. All RMSEs are also reported in Table~\ref{table1}. The theoretical convergence order $1$, established in Theorem~\ref{thm:convergence_FHN}, is confirmed numerically. 
The S splitting yields the smallest RMSEs among all considered numerical methods. The RMSEs of the LT method lie slightly above those obtained under the TrEM and DTrEM methods, which are identical (up to the reported precision). The RMSEs of the tamed Euler-Maruyama methods are larger than those obtained under the splitting, TrEM and DTrEM methods. For the DTEM method we only observe a convergence of order $1/2$, in agreement with the observations in \cite{Kelly2018,Tretyakov2012}. The PTrEM method yields the largest RMSEs. However, we observe that for smaller values of $\sigma_i$, $i=1,2$, the method improves (see also Appendix \ref{app_cubic}, where the impact of the noise intensity on the performance of that method is discussed).

\begin{table}[H]
	{\small  
		\caption{RMSE \eqref{eq:RMSE} for different values of $\Delta$. All model parameters are set to $1$, $X_0=(0,0)^\top$ and $T=5$.}
		\vspace{-0.7cm}
		\label{table1}
		\begin{center}
			\scalebox{1.0}{
				\begin{tabular}{|c|c|c|c|c|c|c|c|}
					\hline 
					$\Delta$ & $\textrm{S}$ & $\textrm{LT}$ & $\textrm{TEM}$ & $\textrm{DTEM}$ & $\textrm{TrEM}$ & $\textrm{DTrEM}$ & $\textrm{PTrEM}$ \\ \hline 
					$2^{-6}$ & $0.01320$ & $0.01751$ & $0.03628$ & $0.27208$ & $0.01667$ & $0.01667$ & {$2.23688$}  \\
					
					$2^{-7}$ & $0.00659$ & $0.00871$ & $0.01795$ & $0.20075$ & $0.00824$ & $0.00824$ & {$1.41733$} \\
					
					$2^{-8}$ & $0.00323$ & $0.00431$ & $0.00889$ & $0.14726$ & $0.00410$ & $0.00410$ & {$0.83296$} \\
					
					$2^{-9}$ & $0.00161$ & $0.00213$ & $0.00438$ & $0.10639$ & $0.00204$ & $0.00204$ & {$0.45305$} \\
					
					$2^{-10}$ & $0.00079$ & $0.00106$ & $0.00213$ & $0.07664$ & $0.00101$ & $0.00101$ & {$0.22307$} \\
					
					$2^{-11}$ & $0.00039$ & $0.00053$ & $0.00099$ & $0.05446$ & $0.00050$ & $0.00050$ & {$0.10771$} \\
					
					$2^{-12}$ & $0.00020$ & $0.00027$ & $0.00043$ & $0.03877$ & $0.00027$ & $0.00027$ & {$0.05273$} \\
					\hline
			\end{tabular}}
	\end{center}}
\end{table}  

\vspace{-0.7cm}
\subsection{Preservation of neuronal spiking dynamics: amplitudes and frequencies}
\label{sec:7:2_FHN}

In the following, we analyse the ability of the considered methods to preserve the qualitative neuronal spiking dynamics of the FHN model. In particular, we investigate whether the amplitudes and frequencies of the neuronal oscillations are kept when increasing the time step~$\Delta$. Throughout this and the next section, we omit the DTEM method \eqref{DTEM} as it yields a performance comparable to that of the TEM method \eqref{eq:EM_Tamed}. Moreover, we set $\beta=\sigma_1=0.1$ and $\sigma_2=0.2$, and consider different values for $\gamma$ and $\epsilon$. These parameters are of particular interest, because they regulate the spiking intensity of the neuron and separate the time scale of the two model components, respectively. When $\epsilon$ is small, both variables evolve on different time scales. This situation is often referred to as ``stiff'' case, while larger values of $\epsilon$ refer to the 	``nonstiff'' case, see, e.g., \cite{Stern2020}. Furthermore, these parameters determine the value of $\norm{e^{A\Delta}}$, and thus the validity of Theorem \ref{thm:Lyapunov_discrete}. Since the true process is not available, all reference paths are obtained under the TEM method~\eqref{eq:EM_Tamed}, using the small time step $\Delta=2\cdot 10^{-5}$. Also in this case, the choice of the scheme used to simulate the reference paths does not affect the results of the experiments. Moreover, note that all paths are generated using the same set of pseudo random numbers in each~example.

In the following, the focus lies on the $V$-component of the process solving SDE \eqref{FHN}, modelling the membrane voltage, which can be experimentally recorded with intracellular measurements. Similar results are obtained for the $U$-component.

In Figure \ref{fig:paths}, we report paths of the $V$-component of the FHN model generated under different values of the time step $\Delta$. An increase in $\gamma$ leads to an increase in the frequency of the oscillations, and thus in the number of released spikes. Both splitting methods yield almost overlapping paths as $\Delta$ increases, preserving thus the qualitative dynamics of the model, independently of the choice of the intensity parameter $\gamma$. In contrast, the TEM method underestimates the frequency and overestimates the amplitude of the neuronal oscillations as $\Delta$ increases, for both values of $\gamma$ under consideration. For similar observations regarding tamed methods, we refer to \cite{Kelly2017,Kelly2018}. Note also that the paths btained under the TEM method already start deviating from the reference paths for $\Delta=2\cdot 10^{-3}$, performing thus worse than the TrEM and DTrEM methods. Since $\epsilon=0.05$, the quantity $\Delta^*$ required for the PTrEM method (see Section~\ref{sec:7:1}) equals $1/\epsilon^{-5}=3.125 \cdot 10^{-7}$. We then observe that this method produces the desired paths only for very small time steps ($\Delta=2 \cdot 10^{-8}$ in Figure \ref{fig:paths}) and fails for the other values of $\Delta$ under consideration.

\begin{figure}[t]
	\begin{centering}
		\includegraphics[width=1.0\textwidth]{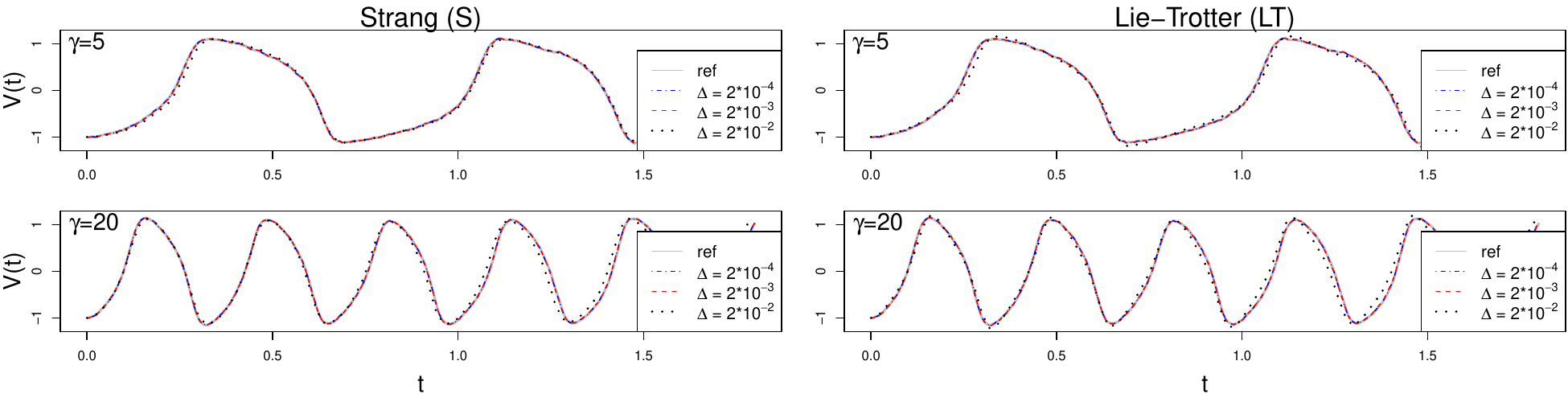}\\
		\includegraphics[width=1.0\textwidth]{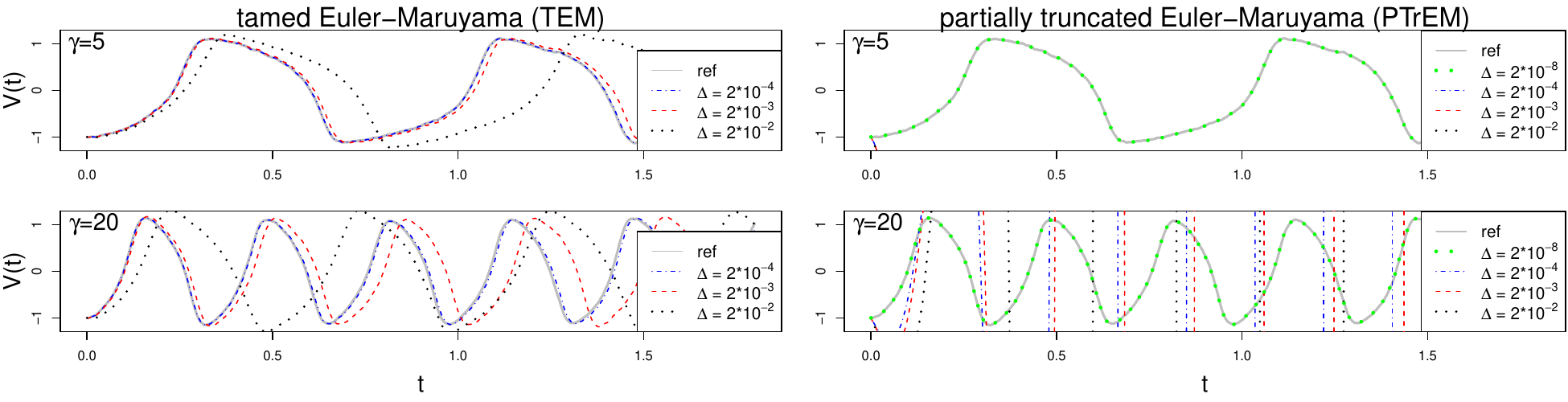}\\
		\includegraphics[width=1.0\textwidth]{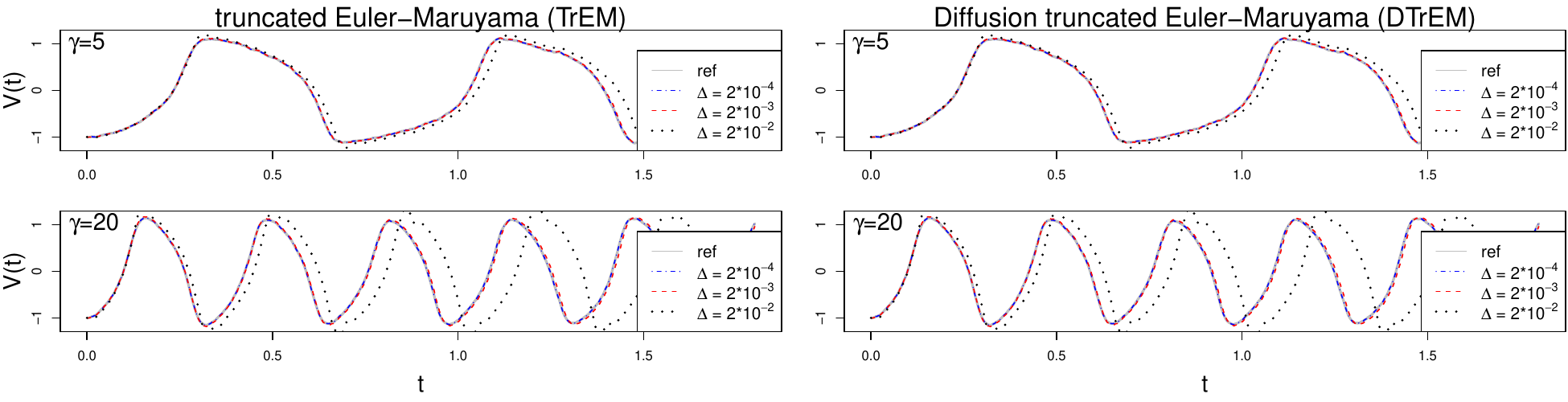}
		\caption{Paths of the $V$-component of the FHN model \eqref{FHN} simulated under the considered numerical methods for $X_0=(-1,0)^\top$, $\beta=\sigma_1=0.1$, $\sigma_2=0.2$, $\epsilon=0.05$, two different values of $\gamma$ and increasing time step $\Delta$. Paths of the PTrEM method are also generated under a smaller time step than used for the other methods, i.e., $\Delta=2\cdot 10^{-8}$. All paths correspond to the same random realisation.}
		\label{fig:paths}
	\end{centering}
\end{figure}

\begin{figure}[t]
	\begin{centering}
		\includegraphics[width=1.0\textwidth]{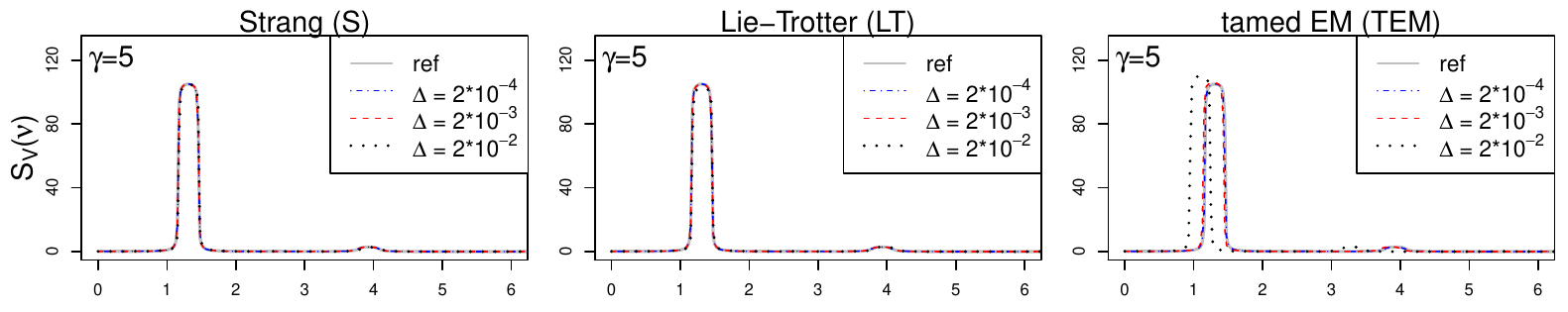}\\
		\includegraphics[width=1.0\textwidth]{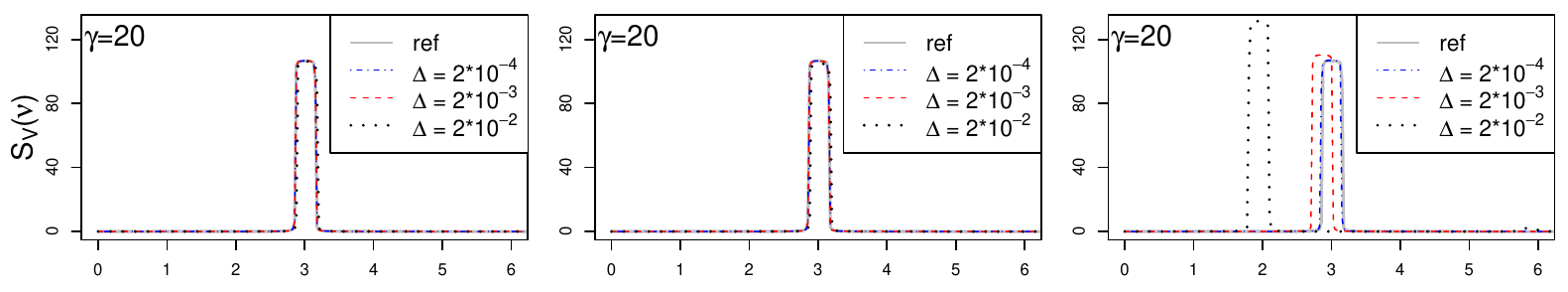}\\
		\includegraphics[width=1.0\textwidth]{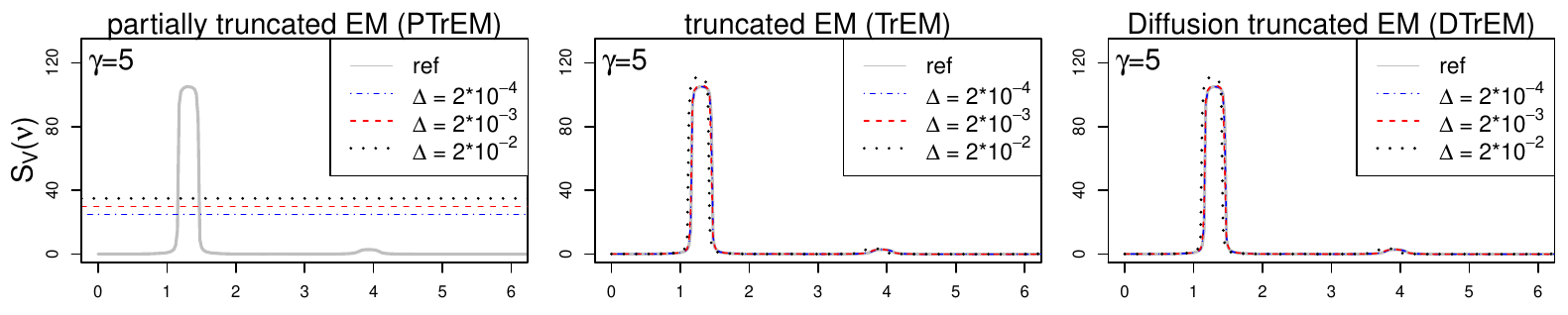}\\
		\includegraphics[width=1.0\textwidth]{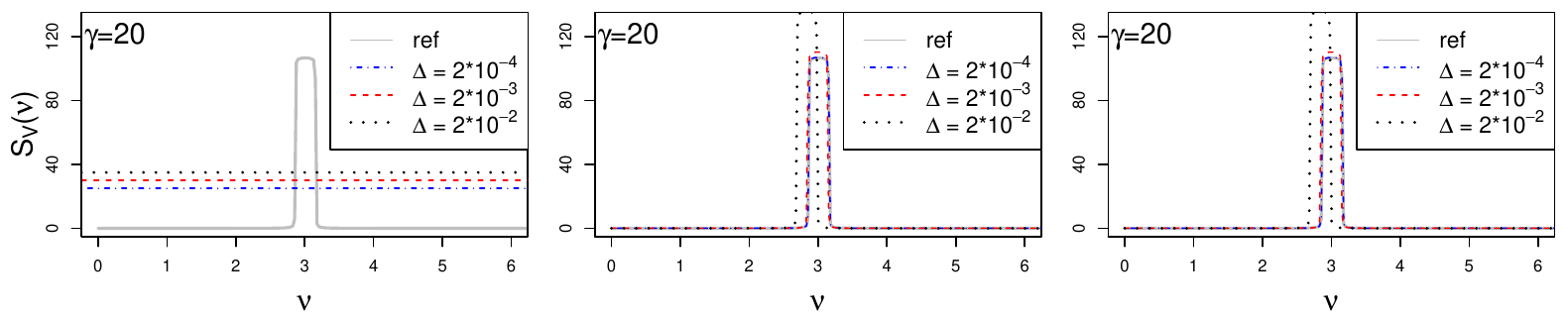}
		\caption{Estimates of the spectral density \eqref{spectrum} of the $V$-component of the FHN model \eqref{FHN} obtained under the considered numerical methods for $X_0=(0,0)^\top$, $\beta=\sigma_1=0.1$, $\sigma_2=0.2$, $\epsilon=0.05$, two different values of $\gamma$ and increasing time step $\Delta$.}
		\label{fig:specDens}
	\end{centering}
\end{figure}

For a deeper investigation of the neuronal spiking dynamics, we consider the spectral density of the $V$-component, which takes into account its autocovariance, and thus the dependence of the membrane voltage on previous epochs. It is given by
\begin{equation}\label{spectrum}
	S_V(\nu)=\mathcal{F}\left\{ r_V \right\} (\nu)=\int\limits_{-\infty}^{\infty} r_V(\tau)e^{-i 2 \pi \nu \tau} \ d\tau, 
\end{equation}
where $\mathcal{F}$ denotes the Fourier transformation, $r_V$ the autocovariance function of $(V(t))_{t \in [0,T]}$ and the frequency $\nu$ can be interpreted as the number of oscillations in one time unit. We estimate the spectral density $S_V(\nu)$ with a smoothed periodogram estimator, see, e.g., \cite{Buckwar2019,Quinn2014}, based on paths generated over the time interval $[0,10^3]$.  We use the R-function \texttt{spectrum} and set the required smoothing parameter to \texttt{span}$=0.3T$.

The estimated spectral densities obtained under different values of $\gamma$ and different choices of the time step $\Delta$ are reported in Figure \ref{fig:specDens}. As desired, for a fixed $\gamma$, all spectral densities estimated from the paths generated under the splitting schemes are almost overlapping as $\Delta$ increases. In contrast, the frequency $\nu$ estimated under the Euler-Maruyama type methods decreases as $\Delta$ increases, and the height of the peaks, carrying information about the amplitude of the neuronal oscillations, increases with $\Delta$. Their performance deteriorates as $\gamma$ increases, the TrEM and DTrEM methods yielding better results than the TEM method. For the considered values of $\Delta$, the spectral densities based on the PTrEM method cannot be derived, because this method produces ``NaN'' values. This is indicated by the horizontal lines in the bottom left panels. Note also that the estimated frequencies are in agreement with those deduced from Figure \ref{fig:paths}. 

\begin{figure}
	\begin{centering}
		\includegraphics[width=1.0\textwidth]{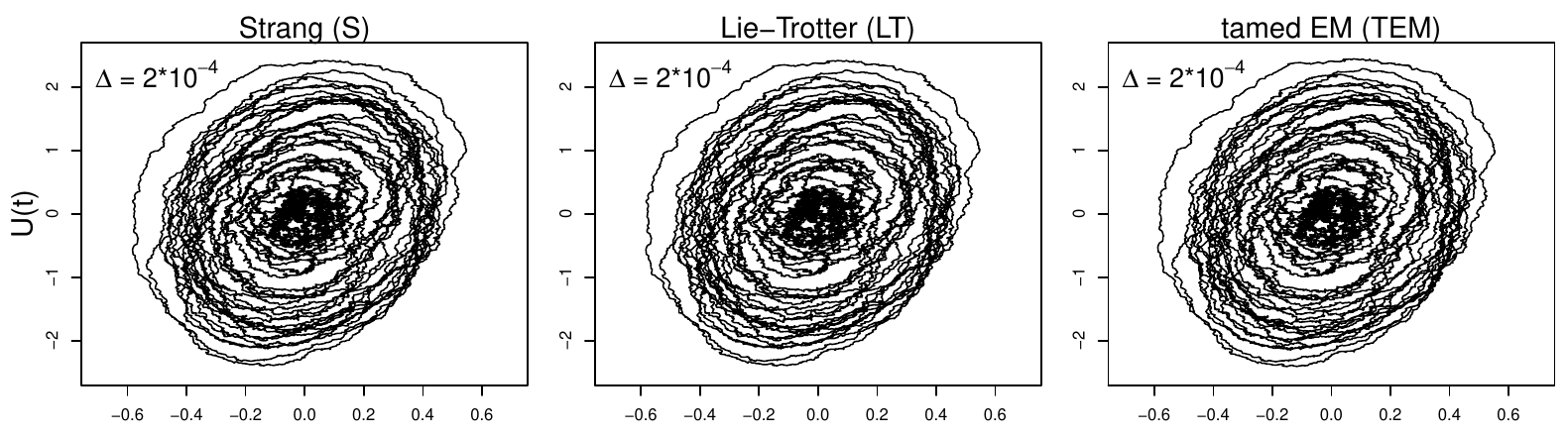}\\
		\includegraphics[width=1.0\textwidth]{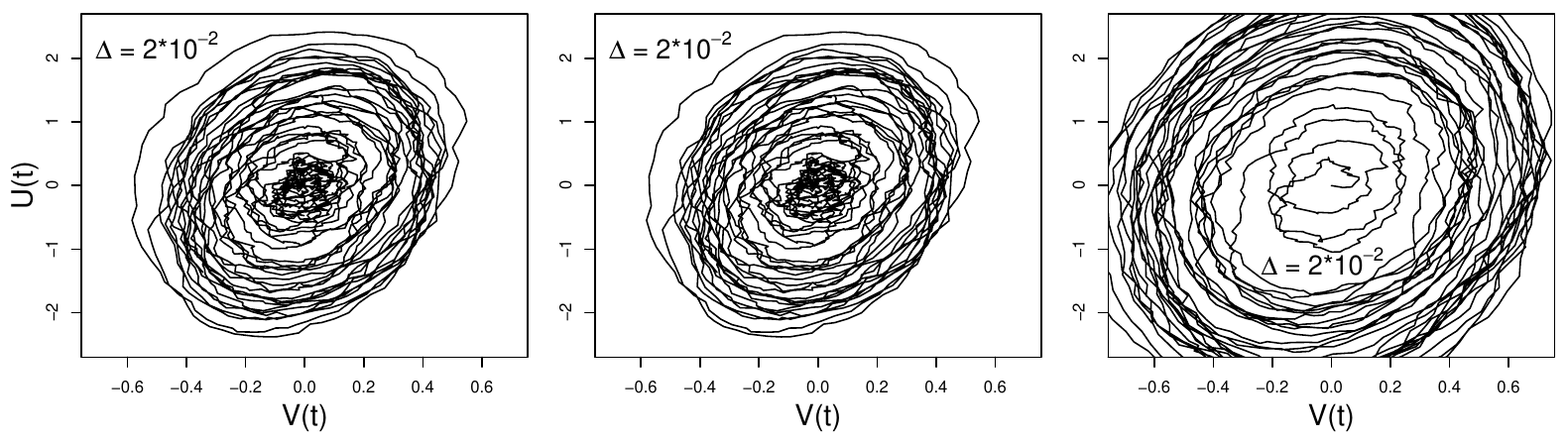}\\
		\includegraphics[width=1.0\textwidth]{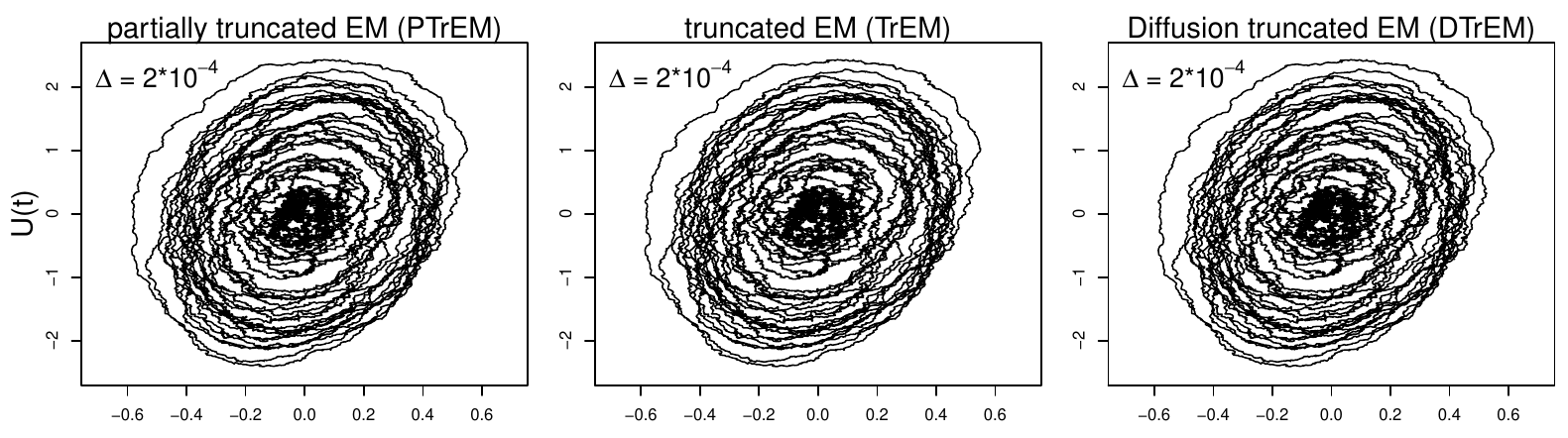}\\
		\includegraphics[width=1.0\textwidth]{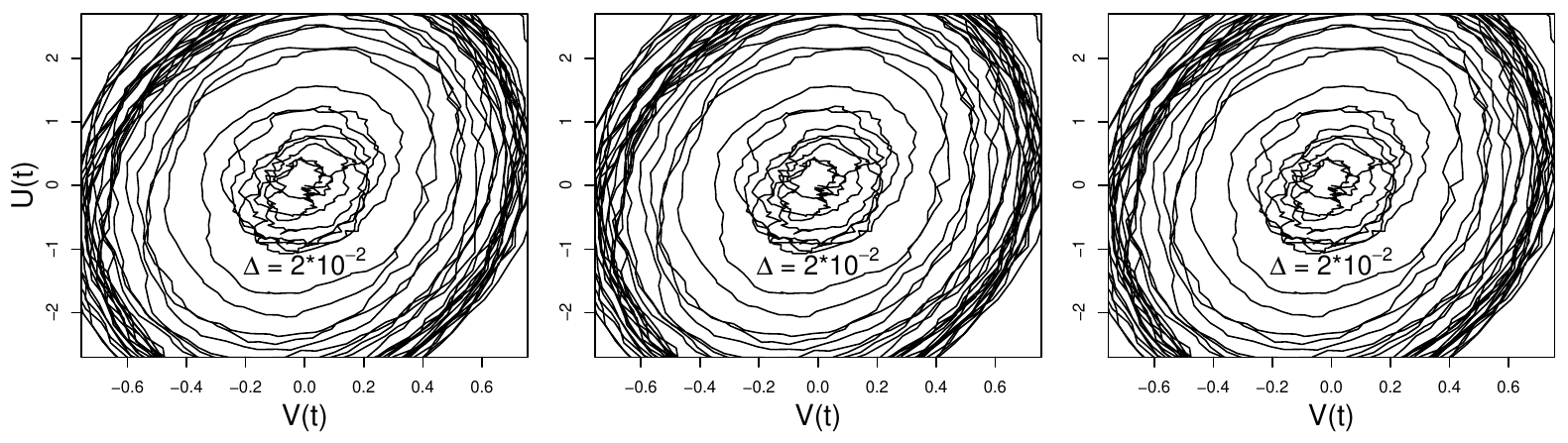}
		\caption{Phase portraits of the FHN model \eqref{FHN} simulated under the considered numerical methods for $X_0=(0,0)^\top$, $\beta=\sigma_1=0.1$, $\sigma_2=0.2$, $\epsilon=1$, $\gamma=20$ and increasing time step~$\Delta$. All paths correspond to the same random realisation.}
		\label{fig:phase}
	\end{centering}
\end{figure}

Moreover, the Euler-Maruyama type methods perform even worse in terms of second moment (amplitude) preservation when the parameter $\epsilon$ is increased, while the splitting methods preserve the qualitative behaviour of the model. This is illustrated in Figure~\ref{fig:phase}, where we increase $\epsilon$ to $1$ (the quantity $\Delta^*$ introduced  for the PTrEM method in Section \ref{sec:7:1} thus equals $1$), fix $\gamma=20$ and report phase portraits of the system obtained under the different numerical methods for $\Delta=2\cdot10^{-4}$ and $\Delta=2\cdot10^{-2}$. Again, the splitting methods preserve the behaviour of the process $(X(t))_{t \in [0,T]}$ as $\Delta$ increases, while the Euler-Maruyama type methods produce larger orbits, overshooting the second moment of the process.

In addition, we investigate the ability of the considered numerical methods to approximate the underlying invariant density of the process $(X(t))_{t \in [0,T]}$. In particular, we estimate the marginal invariant density of the $V$-component of the FHN model with a standard kernel density estimator given by 
\begin{equation}\label{eq:kernel}
	\pi_V(v)=\frac{1}{n\mathfrak{H}}\sum\limits_{i=1}^{n} \mathcal{K}\left( \frac{v-\widetilde{V}(t_i)}{\mathfrak{H}} \right),
\end{equation}
where $\mathfrak{H}$ is a smoothing bandwidth and $\mathcal{K}$ a kernel function \cite{Pons2011}. 
Taking advantage of the ergodicity of the FHN model, the sample $\widetilde{V}(t_i)$, $i=1,\ldots,n$, in \eqref{eq:kernel} is obtained from a long-time simulation of a single path. We use the R-function \texttt{density}, a kernel estimator as described in~\eqref{eq:kernel}.

In Figure \ref{fig:density}, we report the marginal invariant densities of the process $(V(t))_{t \in [0,T]}$ estimated via \eqref{eq:kernel} based on paths generated over the time interval $[0,10^4]$, for $\epsilon=1$, $\gamma=20$ and different values of $\Delta$. Both splitting methods yield reliable estimates for all values of~$\Delta$ under consideration. In contrast, the densities obtained under the Euler-Maruyama type methods already deviate from the desired ones for $\Delta=2\cdot10^{-3}$, and suggest a transition from a unimodal to a bimodal density when $\Delta$ is further increased to $2\cdot 10^{-2}$. It is again visible that the Euler-Maruyama type methods overestimate the second moment, and thus the amplitudes of the process. Similar results are also obtained for the $U$-component. 

\begin{figure}
	\begin{centering}
		\includegraphics[width=1.0\textwidth]{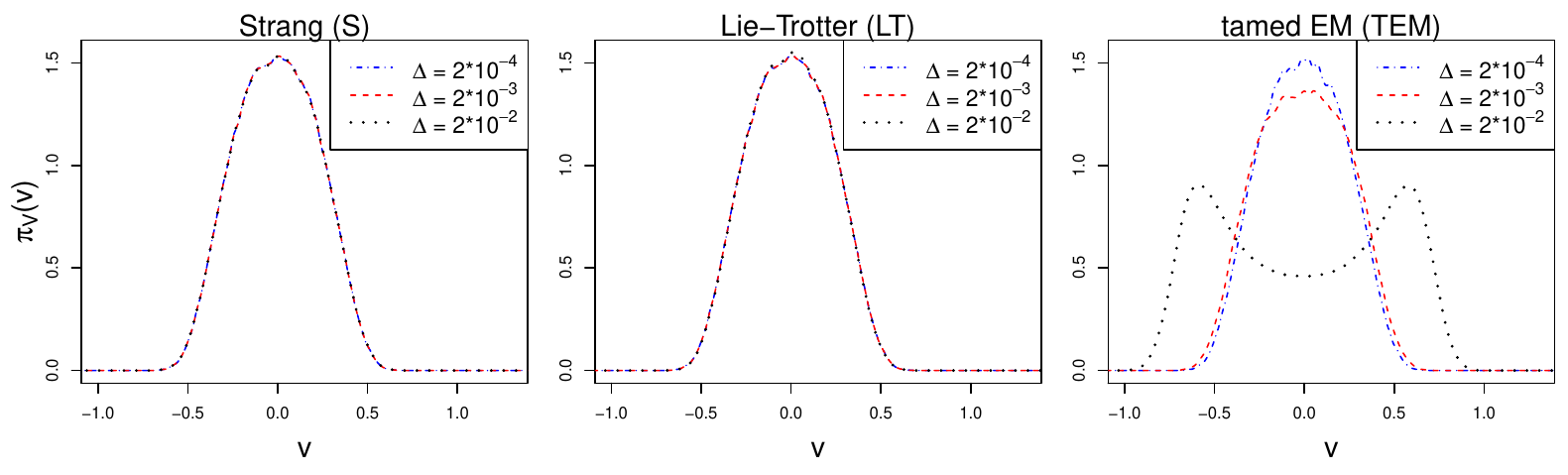}\\
		\includegraphics[width=1.0\textwidth]{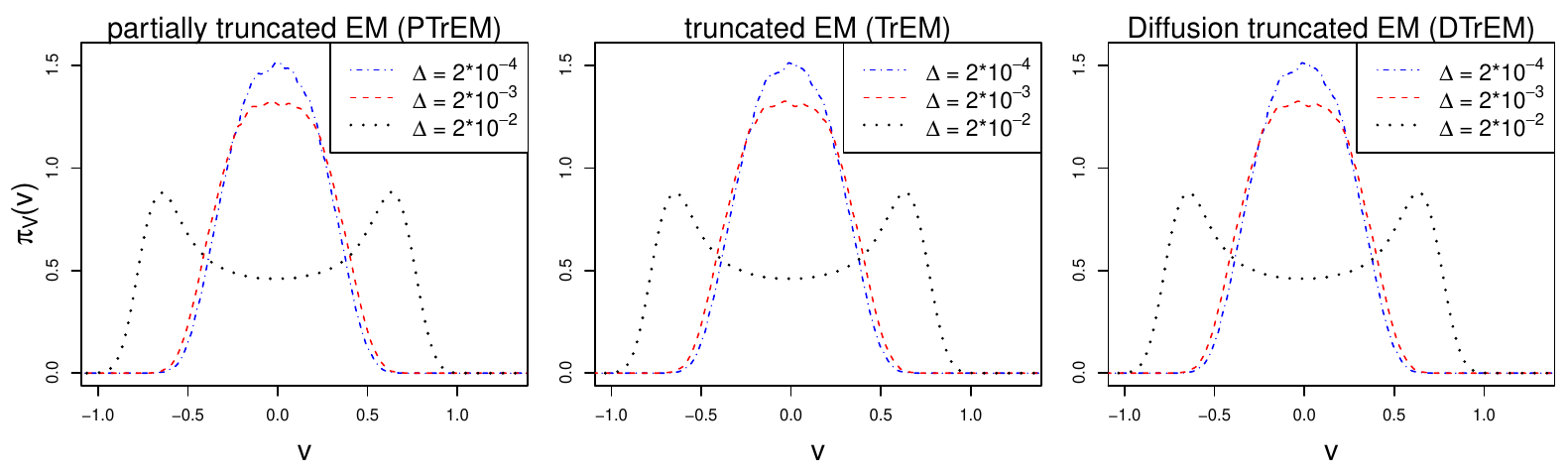}
		\caption{Estimates of the invariant density \eqref{eq:kernel} of the $V$-component of the FHN model \eqref{FHN} obtained under the considered numerical methods for $X_0=(0,0)^\top$, $\beta=\sigma_1=0.1$, $\sigma_2=0.2$, $\epsilon=1$, $\gamma=20$ and increasing time step $\Delta$.}
		\label{fig:density}
	\end{centering}
\end{figure}

\subsection{Impact of the initial condition: preservation of phases}

Finally, we compare the considered numerical methods regarding their sensitivity to changes in the initial condition $X_0$. In particular, we illustrate that, when $V_0$ is large, i.e., when the process starts far away from the mean of the invariant distribution, the considered Euler-Maruyama type methods do not correctly reproduce the phases of the underlying oscillations, even when the time step $\Delta$ is very small. In contrast, the splitting methods are less sensitive to changes in the initial condition. Similar observations are made when $U_0$ is large (figures not shown).

The impact of $V_0$ on the performance of the different numerical methods is shown in Figure \ref{fig:V0} and Figure \ref{fig:U0}, where we report paths of the $V$-component, simulated under $\Delta=2 \cdot 10^{-4}$, $U_0=0$ and different values of $V_0$. The grey reference path is simulated under $\Delta=2\cdot10^{-7}$ using the TEM method \eqref{eq:EM_Tamed}. As before, the results are not influenced by the choice of the numerical method used to generate the reference paths. The underlying parameter values are the same as in Section \ref{sec:7:2_FHN}, choosing $\gamma=5$ and $\epsilon=0.05$ in Figure \ref{fig:V0}, and $\gamma=20$ and $\epsilon=1$ in Figure \ref{fig:U0}. As desired, the splitting methods are barely influenced by $V_0$, even when it is very large, with paths overlapping with the reference paths for all $t$ under consideration. In contrast, when $V_0$ is large, the Euler-Maruyama type methods introduce a delay in when the generated paths reach the oscillatory dynamics, this behaviour deteriorating as $V_0$ increases. Moreover, they also do not preserve the phases of the oscillations, introducing a shift. In Figure \ref{fig:V0}, the DTrEM method reaches the correct oscillatory dynamics, though shifted, almost as fast as the splitting methods for $V_0=10^{4}$, but fails to reach the invariant regime for $V_0=10^3$. In Figure \ref{fig:U0}, it does not enter the invariant regime for both $V_0=10^3$ and $V_0=10^4$. Moreover, spurious oscillations produced by the DTrEM method were obtained for other parameter combinations, as also observed in \cite{Kelly2018,Tretyakov2012}. For $\epsilon=0.05$ (see Figure~\ref{fig:V0}), the PTrEM method does not produce the desired paths, even when $V_0$ is close to the invariant mean. For $\epsilon=1$ (see Figure \ref{fig:U0}), it yields the correct path when $V_0=1$, a path which initially deviates from the others when $V_0=3$, and produces high-amplitude oscillations, not entering the invariant regime, when $V_0=10^3$ and $V_0=10^4$. Therefore, the PTrEM method reacts very sensitively to the choice of $X_0$, this undesired behaviour being also observed for the cubic model problem \eqref{eq:Toy} introduced in Section \ref{sec:5_FHN}, see Appendix \ref{app_cubic}.

\begin{remark}
	Note that the only considered combination of $\gamma$ and $\epsilon$ in this section for which Assumption \ref{assum:A_norm}, and thus Theorem \ref{thm:Lyapunov_discrete} holds is $\gamma=1/\epsilon=20$. However, we do not observe a difference in the quality of the splitting methods depending on the combination of these parameters. Intuitively, this is because the underlying linear SDE \eqref{SDE_FHN} with matrix $A$ as in \eqref{eq:A_N_FHN}, i.e., the damped stochastic oscillator, is geometrically ergodic for all $\gamma>0$ and $\epsilon>0$.
\end{remark}

\begin{figure}
	\begin{centering}
		\includegraphics[width=1.0\textwidth]{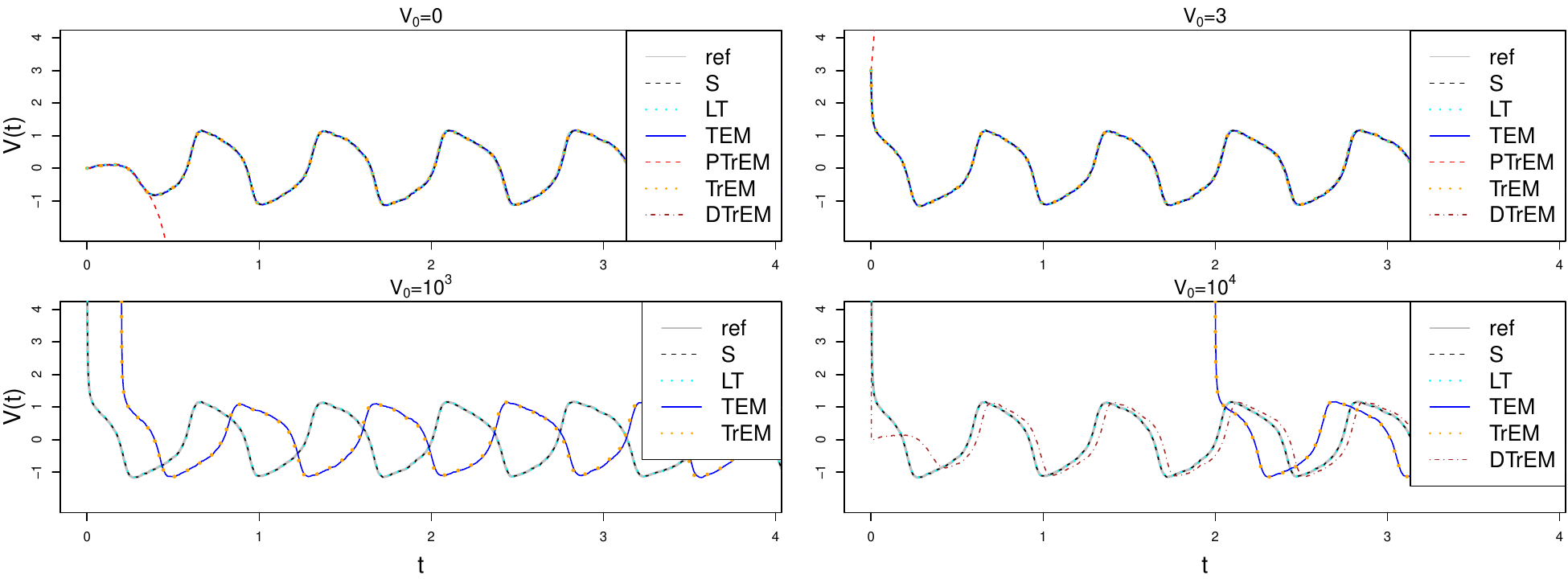}
		\caption{Paths of the $V$-component of the FHN model \eqref{FHN} simulated under the considered numerical methods for different values of $V_0$ ($U_0=0$), $\Delta=2\cdot 10^{-4}$, $\beta=\sigma_1=0.1$, $\sigma_2=0.2$, $\gamma=5$ and $\epsilon=0.05$. The grey reference paths are obtained under $\Delta=2\cdot10^{-7}$ using the TEM method \eqref{eq:EM_Tamed}. All paths correspond to the same random realisation.}
		\label{fig:V0}
	\end{centering}
\end{figure}

\begin{figure}
	\begin{centering}
		\includegraphics[width=1.0\textwidth]{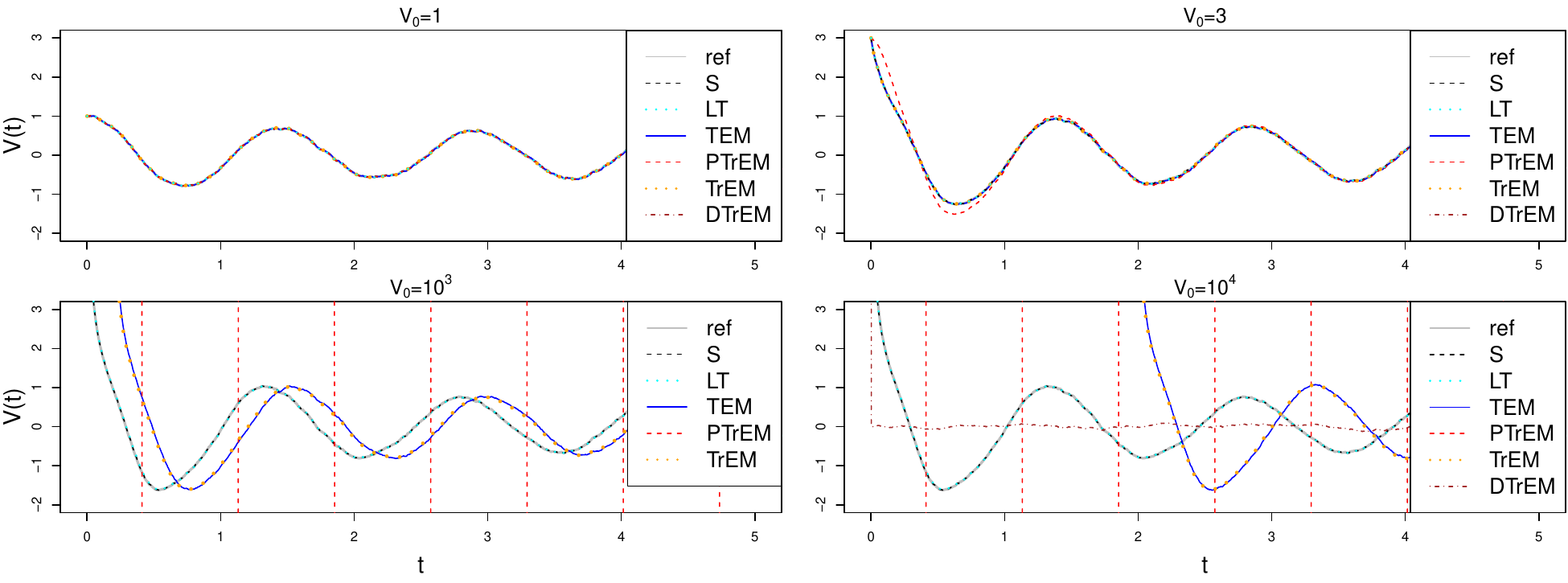}
		\caption{Paths of the $V$-component of the FHN model \eqref{FHN} simulated under the considered numerical methods for different values of $V_0$ ($U_0=0$), $\Delta=2\cdot 10^{-4}$, $\beta=\sigma_1=0.1$, $\sigma_2=0.2$, $\gamma=20$ and $\epsilon=1$. The grey reference paths are obtained under $\Delta=2\cdot10^{-7}$ using the TEM method \eqref{eq:EM_Tamed}. All paths correspond to the same random realisation.}
		\label{fig:U0}
	\end{centering}
\end{figure}

\newpage

\section{Conclusion and discussion}
\label{sec:8_FHN}

We propose a splitting strategy to approximate the solutions of semi-linear SDEs with additive noise and globally one-sided Lipschitz continuous drift coefficients which are allowed to grow polynomially. We prove that the resulting explicit Lie-Trotter splitting method is mean-square convergent of order $1$. In contrast to existing explicit mean-square convergent Euler-Maruyama type methods, which may also achieve a convergence rate of order $1$, the constructed  method preserves important structural properties of the model. 

First, it provides a more accurate approximation of the noise structure of the SDE through the covariance matrix of the exact solution of the stochastic subequation. In particular, while the conditional covariance matrix of Euler-Maruyama type methods only contains the information of the diffusion matrix $\Sigma$, the splitting method also relies on the matrix $A$ in the semi-linear drift. This is particularly beneficial when the SDE is hypoelliptic. Indeed, while the conditional covariance matrix of the existing methods is degenerate in that case, we establish the desired $1$-step hypoellipticity of the constructed splitting method, meaning that it admits a smooth transition density in every iteration step. In particular, the method yields non-degenerate Gaussian transition densities, a feature which is advantageous within likelihood-based estimation techniques, where the existing numerical methods cannot be applied \cite{Ditlevsen2019,Melnykova2018,Pokern2009}. 

Second, Euler-Maruyama type methods do not preserve the geometric ergodicity of the process. As a consequence, they are not robust to changes in the initial condition, yield poor approximations of the underlying invariant distribution, or do not preserve the moments of the process. In contrast, the proposed splitting method is proved to preserve the Lyapunov structure of the SDE, as long as an assumption on the solution $f$ of the deterministic subequation is satisfied and it holds that $\norm{e^{A\Delta}}<1$ for all $\Delta\in(0,\Delta_0]$. If, in addition, the logarithmic norm $\mu(A)<0$, the method is proved to have an asymptotically bounded second moment. In the one-dimensional case, a precise bound of the second moment of the splitting method is derived in closed-form and illustrated on a cubic model problem. We also consider the FHN model, a well known equation used to describe the firing activity of single neurons. The geometric ergodicity of the proposed splitting method applied to this equation is established under a restricted parameter space.

Third, we illustrate on the FHN model that, in contrast to Euler-Maruyama type methods, the proposed splitting method preserves the amplitudes, frequencies and phases of neuronal oscillations, even for large time steps. This may make the method particularly beneficial when used, e.g., to simulate large networks of neurons, or when embedded within simulation-based inference procedures. Besides the Lie-Trotter splitting method, we also consider a method which is based on a Strang composition in our numerical experiments. Both splitting methods perform comparably good throughout, the Strang splitting behaving slightly better in some scenarios. As the considered Euler-Maruyama type methods do converge, their lack of structure preservation becomes less visible when using very small time steps. However, the use of significantly smaller time steps results in drastically higher computational costs, making these methods highly inefficient and, consequently, computationally infeasible within simulation-based inference algorithms, as previously illustrated in~\cite{Buckwar2019}. 

Several generalisations of the considered approach are possible. The proposed splitting strategy can be, e.g., applied to the stochastic Van der Pol oscillator \cite{VanderPol1920,VanderPol1926}, whose investigation leads to similar numerical results. The presented approach may be also applied to SDEs \eqref{eq:semi_linear_SDE} with other types of nonlinearity. In particular, as long as the ODE determined by the function $N$ is exactly solvable (see \cite{Kamke1983} for diverse solution methods) and satisfies some useful conditions, the method presented in this article may be used.
Moreover, the proposed method may be extended to SDEs with multiplicative noise, e.g., to $\Sigma(X(t))=\sigma X(t)$, $\sigma>0$, where the stochastic subequation of the splitting framework corresponds to 
the  geometric Brownian motion. This may be relevant, e.g., for the stochastic Ginzburg-Landau equation arising from the theory of superconductivity \cite{Ginzburg1959,Hutzenthaler2011}.
Furthermore, the investigation of conditions under which the presented results are still valid when the solution of ODE \eqref{ODE_FHN} is not available exactly, constitutes another topic for future~research.

\appendix

\section*{Appendix}

\section{Proof of Proposition \ref{prop:A1_A2_N_toy}}
\label{appA_FHN}

\begin{proof}
	\textbf{Assumption \ref{assum:1}}:
	We start with Assumption \ref{(A1)} and have that
	\begin{equation*}
		\Bigl( N(x)-N(y) \Bigr)(x-y)=(x-y)^2 \Bigl( 1- (x^2+xy+y^2) \Bigr) \leq (x-y)^2.
	\end{equation*}
	Thus, the assumption holds for $c_1=1$, see also Example 1.2.16 in \cite{Humphries2001}.
	
	Now, consider Assumption \ref{(A2)}. We have that
	\begin{eqnarray*}
		\Bigl( N(x)-N(y) \Bigr)^2 
		&\leq& 2(x-y)^2+2(y^3-x^3)^2
		=2(x-y)^2+2(x-y)^2(y^2+xy+x^2)^2 \\ &\leq& 2(x-y)^2+9(x-y)^2(x^4+y^4),
	\end{eqnarray*}
	where we used that $2xy\leq x^2+y^2$ and that ${(3x^2/2+3y^2/2)^2\leq {9}(x^4+y^4)/2}$ in the last inequality. Thus, we obtain that
	\begin{equation*}
		\Bigl( N(x)-N(y) \Bigr)^2 \leq {9}(x-y)^2\Bigl( 1+x^4+{y^4}\Bigr),
	\end{equation*}
	which proves that the assumption holds for $c_2={9}$ and $\chi=3$.
	
	\textbf{Assumption \ref{assum:A_hypo}}: Since $d=1$, this is clear.
	
	\textbf{Assumption \ref{assum:f_additional}}: We prove the statement for $c_3=1/2$. Setting $y=x^2$, it has to be shown that
	\begin{equation*}
		f^2(x;t)=g(y;t):=\frac{y}{e^{-2t}+y(1-e^{-2t})}\leq y+\frac{1}{2}t=:h(y;t), \quad {\forall t \in (0,\Delta_0]}.
	\end{equation*}
	Since $g(y;0)=h(y;0)=y$, it suffices to prove that for any $y\in\mathbb{R}_0^+$ it holds that
	\begin{equation*}
		g'(y;t)=-\frac{2e^{2t}(y-1)y}{\left( 1+(e^{2t}-1)y \right)^2} \leq \frac{1}{2}=h'(y;t), \quad {\forall t \in (0,\Delta_0]},
	\end{equation*}
	where $'$ denotes the derivative with respect to $t$. Consider two cases. First, let $y\notin (0,1)$. Then it holds that $g'(y;t)\leq 0$ for all $t\geq 0$. Second, let $y\in (0,1)$. To prove that $g'(y;t)\leq 1/2$, we determine the global maximum of $g'(y;t)$ with respect to $t$. Solving $g''(y;t)=0$ with respect to $t$, gives that 
	\begin{equation*}
		t_{\textrm{max}}=\frac{1}{2}\log\left(\frac{1}{y}-1\right).
	\end{equation*}
	Noting that $t_{\textrm{max}}$ exists and that $g'(y;t_{\textrm{max}})=1/2$ for any $y\in(0,1)$ proves the result.
	
	\begin{remark}
		If $y\in (0,1)$, it also holds that
		\begin{equation*}
			g'(y;t)\leq 2 e^{2t}y(1-y)\leq \frac{1}{2}e^{2t} \leq \frac{e^{2\Delta_0}}{2}.
		\end{equation*}
		Thus, a simpler argument suffices to prove the statement for $c_3(\Delta_0)=\frac{e^{2\Delta_0}}{2}>\frac{1}{2}$.
	\end{remark}
	
	\textbf{Assumptions \ref{assum:A_norm} and \ref{assum:A_norm_log}}: These statements are satisfied because $A=-1$, and thus the matrix norm $\norm{e^{A\Delta}}=e^{-\Delta}<1$ and the logarithmic norm $\mu(A)=-1<0$.	
\end{proof}

\newpage

\section{Proof of Proposition \ref{prop:A1_A2_N_FHN}}
\label{appC_FHN}

\begin{proof}
	\textbf{Assumption \ref{assum:1}}:
	Denote $x=(v_1,u_1)^\top$ and $y=(v_2,u_2)^\top$. We start with Assumption \ref{(A1)} and have that
	\begin{equation*}
		(N(x)-N(y),x-y)=\frac{1}{\epsilon}(v_1-v_2)^2\left(1-(v_1^2+v_1v_2+v_2^2)\right) 
		\leq\frac{1}{\epsilon}(v_1-v_2)^2 \leq \frac{1}{\epsilon}\norm{x-y}^2. 
	\end{equation*}
	Thus, the assumption holds for $c_1={1}/{\epsilon}$.
	
	Now, consider Assumption \ref{(A2)}. Applying similar arguments as in Appendix \ref{appA_FHN}, we have that
	\begin{equation*}
		\norm{N(x)-N(y)}^2=\left( \frac{1}{\epsilon}(v_1-v_2)+\frac{1}{\epsilon}(v_2^3-v_1^3) \right)^2  \leq \frac{2}{\epsilon^2}(v_1-v_2)^2+\frac{9}{\epsilon^2}(v_1-v_2)^2(v_1^4+v_2^4).
	\end{equation*} 
	Using that $(v_1-v_2)^2\leq\norm{x-y}^2$ and that $v_1^4+v_2^4\leq\norm{x}^4+\norm{y}^4$, we finally obtain that
	\begin{eqnarray*}
		\norm{N(x)-N(y)}^2&\leq&\frac{9}{\epsilon^2}\norm{x-y}^2\left(1+\norm{x}^4+\norm{y}^4 \right).
	\end{eqnarray*}
	Thus, the assumption holds for $c_2=9/\epsilon^2$ and $\chi=3$.
	
	\textbf{Assumption \ref{assum:A_hypo}}: Condition \eqref{eq:cond_hypo} holds for the linear SDE \eqref{SDE_FHN}, since
	\begin{equation*}
		\partial_u(Ax)_1\sigma_2=-\frac{\sigma_2}{\epsilon}\neq 0.
	\end{equation*}
	Thus, the equation is hypoelliptic.
	
	\textbf{Assumption \ref{assum:f_additional}}: 	We have that
	\begin{equation*}
		f(x;\Delta)=(f_1(v;\Delta),f_2(u;\Delta))^\top.
	\end{equation*}
	Consider the $V$-component. The fact that, for any $v\in\mathbb{R}$ it holds that 
	\begin{equation*}
		f_1^2(v;\Delta)\leq v^2+\frac{1}{2\epsilon} \Delta \quad \forall \ \Delta \geq 0,
	\end{equation*}
	can be proved in the same way as in Appendix \ref{appA_FHN}. 
	Regarding the $U$-component, by assumption, we have that
	\begin{equation*}
		f_2^2(u;\Delta)=u^2.
	\end{equation*}
	Thus, 
	\begin{equation*}
		\norm{f(x;\Delta)}^2=f_1^2(v;\Delta)+f_2^2(v;\Delta)\leq v^2+u^2+\frac{1}{2\epsilon}\Delta =\norm{x}^2+\frac{1}{2\epsilon}\Delta,
	\end{equation*}
	which proves the statement for $c_3=1/(2\epsilon)$.
	
	\textbf{Assumption \ref{assum:A_norm}}: Recall that 
	\begin{equation*}
		\norm{e^{A\Delta}}=\sqrt{\lambda_{\textrm{max}}\left( (e^{A\Delta})^\top (e^{A\Delta}) \right)},
	\end{equation*}
	and define $B:=(e^{A\Delta})^\top (e^{A\Delta})$.
	It suffices to prove that $\lambda_{\textrm{max}}\left( B \right)<1$ for all $\Delta \in (0,\Delta_0]$.
	
	Since by assumption $\gamma=1/\epsilon$, $\kappa$ defined in \eqref{kappa} becomes $\kappa=4\gamma^2-1$. When $\kappa=0$, this condition is equivalent to $\gamma=1/2$. In this case, the eigenvalues of $B$ are given by
	\begin{eqnarray*}
		\lambda_1(\Delta)=\frac{1}{2}e^{-\Delta}\left( 2+\Delta^2 -\sqrt{\Delta^2(4+\Delta^2)} \right) \leq 
		\lambda_2(\Delta) =\frac{1}{2}e^{-\Delta}\left( 2+\Delta^2 +\sqrt{\Delta^2(4+\Delta^2)} \right).
	\end{eqnarray*}
	It holds that $\lambda_2'(\Delta)<0$ for all $\Delta>0$, where $'$ denotes the derivative with respect to $\Delta$. Thus, $\lambda_2(\Delta)$ is strictly decreasing in $\Delta$. Noting that $\lambda_2(0)=1$ implies the statement.
	
	When $\kappa<0$, $\gamma<1/2$. In this case, the eigenvalues of $B$ are given by
	\begin{eqnarray*}
		\lambda_1(\Delta,\gamma)&=& \frac{e^{-\Delta}}{\kappa}\left( 4\gamma^2-\cosh(\sqrt{-\kappa}\Delta)-\sqrt{2 \left[1-8\gamma^2+\cosh(\sqrt{-\kappa}\Delta)\right] \sinh^2(\sqrt{-\kappa}\Delta/2) } \right),  \\
		\lambda_2(\Delta,\gamma)&=&\frac{e^{-\Delta}}{\kappa}\left( 4\gamma^2-\cosh(\sqrt{-\kappa}\Delta)+\sqrt{ 2\left[1-8\gamma^2+\cosh(\sqrt{-\kappa}\Delta)\right] \sinh^2(\sqrt{-\kappa}\Delta)/2 } \right),
	\end{eqnarray*}
	where $\lambda_1(\Delta,\gamma)\leq \lambda_2(\Delta,\gamma)$ for all $\Delta>0$ and $\gamma<1/2$. For $\gamma<1/2$ arbitrary, but fixed, the partial derivative of $\lambda_2(\Delta,\gamma)$ with respect to $\Delta$ exists and satisfies
	\begin{equation*}
		\frac{\partial}{\partial \Delta} \lambda_2(\Delta,\gamma)<0, \quad  \forall \ \Delta \in (0,\Delta_0].
	\end{equation*}
	Thus, the function $\lambda_2(\Delta,\gamma)$ is strictly decreasing in $\Delta$. Moreover, we have that $\lambda_2(0,\gamma)=1$ for any $\gamma$, which implies the statement.
	
	When $\kappa>0$, $\gamma>1/2$. In this case, the eigenvalues of $B$ are given by
	\begin{eqnarray*}
		\lambda_1(\Delta,\gamma)&=& \frac{e^{-\Delta}}{\kappa}\left( 4\gamma^2-\cos(\sqrt{\kappa}\Delta)-\sqrt{2 \left[-1+8\gamma^2-\cos(\sqrt{\kappa}\Delta)\right] \sin^2(\sqrt{\kappa}\Delta/2) } \right),  \\
		\lambda_2(\Delta,\gamma)&=&\frac{e^{-\Delta}}{\kappa}\left( 4\gamma^2-\cos(\sqrt{\kappa}\Delta)+\sqrt{ -1+8\gamma^2-8\gamma^2\cos(\sqrt{\kappa}\Delta) +\cos^2(\sqrt{\kappa}\Delta) } \right).
	\end{eqnarray*}
	Again, we observe that $\lambda_1(\Delta,\gamma)\leq 	\lambda_2(\Delta,\gamma)$ for all $\Delta>0$ and $\gamma>1/2$. Consider $\gamma>1/2$ arbitrary, but fixed. Moreover, define $I_\Delta:=\{ 2\pi l/\sqrt{\kappa}, \ l \in \mathbb{N} \}$. Since $\cos(2\pi l)=1$ and $\sin(\pi l)=0$, we have that
	\begin{equation*}
		\lambda_1(\gamma,\Delta)=\lambda_2(\gamma,\Delta)=e^{-\Delta}<1, \quad \forall \ \Delta \in I_\Delta.
	\end{equation*}
	Let now $\Delta \in (0,\infty)\backslash I_\Delta$. For those values of $\Delta$, the partial derivative of $\lambda_2(\Delta,\gamma)$ with respect to $\Delta$ exists. In particular, we have that 
	\begin{equation*}
		\frac{\partial}{\partial \Delta} \lambda_2(\Delta,\gamma)<0, \quad  \forall \ \Delta \in (0,\infty)\backslash I_\Delta.
	\end{equation*}
	Thus, for a fixed $\gamma$, the function $\lambda_2(\Delta,\gamma)$ is strictly decreasing in $\Delta$. Noting that
	$\lambda_2(0,\gamma)=1$ for any $\gamma$ implies the statement.	
\end{proof}

\section{Numerical experiments for the cubic model problem}
\label{app_cubic}

Consider the cubic model problem \eqref{eq:Toy} introduced in Section \ref{sec:5_FHN}.
We now illustrate how the choice of $X_0$ influences the behaviour of paths of the ergodic process $X(t)$  simulated under the different numerical methods. If $X_0$ is large compared to the invariant mean, the standard Euler-Maruyama method \eqref{EM_FHN} produces paths which are computationally pushed to $+/-$ infinity within a few iteration steps, even for very small values of $\Delta$. This is not the case for the tamed/truncated variants of this method. However, they may also react sensitively to $X_0$, even for small $\Delta$. This is illustrated in Figure~\ref{fig:paths_toy}, where we report paths of SDE \eqref{eq:Toy} generated for different values of $X_0$, using $\Delta=10^{-4}$ and $\sigma=1/2$.
The grey reference paths are simulated under $\Delta=10^{-7}$ using the TEM method~\eqref{eq:EM_Tamed}. The choice of the reference method does not change the reported results, and all paths are generated under the same underlying pseudo random numbers. Note that, the DTEM method \eqref{DTEM} is not reported in Figure~\ref{fig:paths_toy}, because it shows a performance comparable to that of the TEM method~\eqref{eq:EM_Tamed}. As desired, the paths obtained under the splitting methods \eqref{SP1_FHN} and \eqref{SP3_FHN} are not deterred by large values of $X_0$, and overlap with the reference path for all values of $X_0$ under consideration. In contrast, for large values of $X_0$, the Euler-Maruyama type methods introduce a delay in when the respective paths reach the reference path. This behaviour deteriorates as $X_0$ increases. The path obtained under the PTrEM method \eqref{PTrEM} initially deviates from the desired one, even when $X_0=5$, not reaching the reference path for the values of $t$ under consideration for $X_0=10^4$ and $X_0=3\cdot 10^4$. Note also that, for some values of $X_0$, we observe that the DTrEM method~\eqref{DTrEM} may produce spurious oscillations (figures not shown). See \cite{Kelly2018,Tretyakov2012}, where such a behaviour has also been observed.

\begin{figure}[H]
	\begin{centering}
		\includegraphics[width=1.0\textwidth]{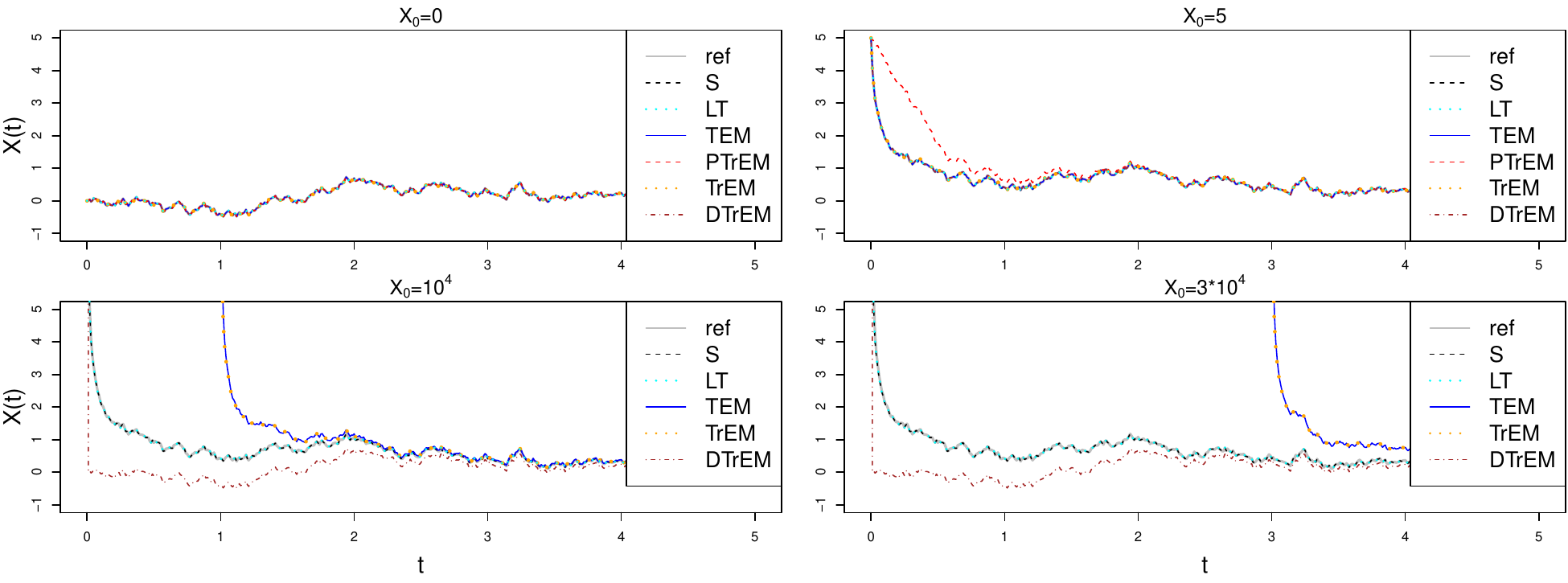}	
		\caption{Paths of SDE \eqref{eq:Toy} simulated under the considered numerical methods for different values of $X_0$, $\Delta=10^{-4}$ and $\sigma=1/2$. The grey reference paths are obtained under $\Delta=10^{-7}$ using the TEM method \eqref{eq:EM_Tamed}. All paths correspond to the same random realisation.}
		\label{fig:paths_toy}
	\end{centering}
\end{figure}

\begin{figure}[H]
	\begin{centering}
		\includegraphics[width=1.0\textwidth]{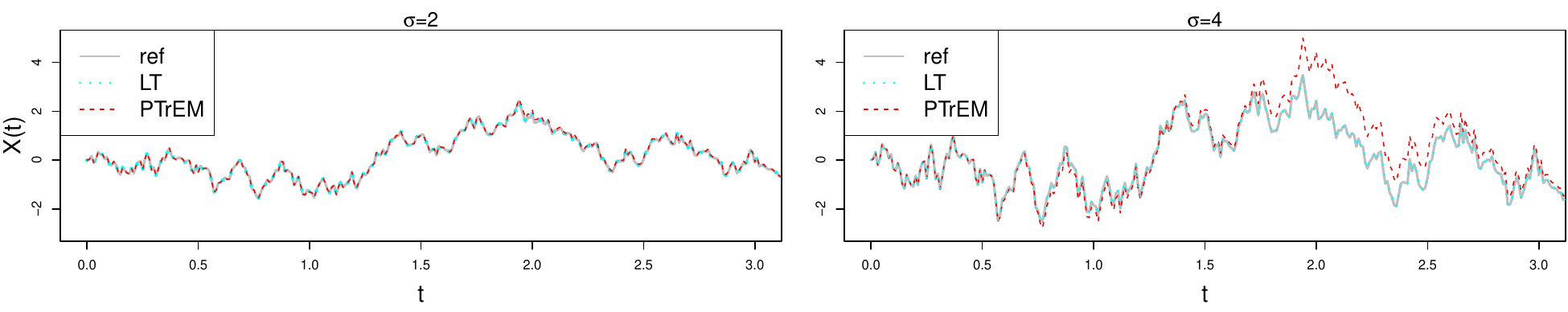}	
		\caption{Paths of SDE \eqref{eq:Toy} simulated under the LT and PTrEM methods for $X_0=0$, $\Delta=10^{-4}$ and different values of $\sigma$. The grey reference paths are obtained under $\Delta=10^{-7}$ using the TEM method \eqref{eq:EM_Tamed}. All paths correspond to the same random realisation.}
		\label{fig:paths_toy_sig}
	\end{centering}
\end{figure}	

In addition, we observe that the PTrEM method \eqref{PTrEM} may also produce paths which deviate from the desired ones for larger values of the noise parameter $\sigma$. This is illustrated in Figure~\ref{fig:paths_toy_sig}, where we report paths of SDE \eqref{eq:Toy} generated under $X_0=0$, $\Delta=10^{-4}$ and different values of $\sigma$. While the paths obtained under the LT splitting \eqref{SP1_FHN} (the same is observed for all other schemes except for the PTrEM method) overlap with the reference paths for both values of $\sigma$ under consideration, the PTrEM method produces  a path which deviates from the desired one when $\sigma=4$ (right panel). This behaviour deteriorates as $\sigma$ increases.

%%%%%%%%%%%%%%%%%%%%%%%%%%%%%%%%%%%%%%%%%%%%%%%%%%%%%%%%%%%%%%%%%%%%%%%%%%%%
% Literaturverzeichnis

% Bibliography (alphabetisch nach autor sortiert)

\newpage
%\bibliography{lit}{} 
%\bibliographystyle{plain}

\end{document}